\documentclass[10pt,a4]{article}
\usepackage{latexsym}
\usepackage{amsmath}
\usepackage{amssymb}
\usepackage{amsthm}
\usepackage{amscd}
\usepackage{mathrsfs}
\usepackage[all]{xy}

\newtheorem{theorem}{Theorem}[section]
\newtheorem{lemma}[theorem]{Lemma}
\newtheorem{corollary}[theorem]{Corollary}
\newtheorem{proposition}[theorem]{Proposition}

\theoremstyle{definition}

\newtheorem{remark}[theorem]{Remark}
\newtheorem{definition}[theorem]{Definition}

\theoremstyle{remark}
\newtheorem{notation}{Notation}

\makeatletter

\@addtoreset{equation}{section}
\makeatother

\renewcommand{\eqref}[1]{(\ref{#1})}

\renewcommand{\bigskip}{\vspace{0.2cm}}

\begin{document}

\title{Elementary computation of the stable reduction of the Drinfeld modular curve $X(\pi^2)$}

\maketitle

\begin{center}
{\sc Takahiro Tsushima}\\
{\it
Graduate School of Mathematical Sciences,
The University of Tokyo, 3-8-1 Komaba, Meguro-ku
Tokyo 153-8914, JAPAN}
\\
E-mail: tsushima@ms.u-tokyo.ac.jp
\end{center}
\maketitle
\begin{abstract}
In \cite{W3}[Theorem 1.1], Jared Weinstein proves that, in the stable reduction 
of the Lubin-Tate space $\mathcal{X}(\pi^n),$
 all irreducible components admit a purely 
inseparable map to
one of the following four curves; 
the projective line, 
the curve with the Artin-Schreier affine model $a^q-a=t^2,$
 the Deligne-Lusztig curve with affine model $x^qy-xy^q=1$
and the curve with affine model $a^q-a=t^{q+1}.$ 
To prove this, he uses non-abelian Lubin-Tate theory for ${\rm GL}_2$ 
in \cite{W3}[Theorem 3.3]
 and the Bushnell-Kutzko type theory.
In this paper, we precisely determine defining equations  
 of all irreducible components in the stable reduction of the
 Drinfeld modular curve of level $2.$
 Our method is purely local, explicit and elementary with using blow-up.
 As a corollary, we determine the inertia action and ${\rm GL}_2$-action
 on each components
 in the stable reduction of $\mathcal{X}(\pi^2)$ explicitly.
\end{abstract}

\section{Introduction}
\noindent
Let $F$ be a non-archimedean local field with uniformizer $\pi.$
Let $\mathbf{C}$ be the completion of a fixed algebraic closure of $F.$ 
By a model for a scheme $X$ over $F,$
we mean a scheme $\mathcal{X}$ over the ring of integers
$\mathcal{O}_F$ of $F$ such that $X \simeq \mathcal{X} \otimes_{\mathcal{O}_F} F.$ 
When a curve $C$ over $F$
does not have a model with good reduction
over $\mathcal{O}_F,$
it may have the ``next best thing,"
i.e., {\it a stable model}.
The stable model is unique up to isomorphism if it exists,
and it does over the ring of integers in
some finite extension of $F,$
as long as the genus of the curve is at least $2,$
which is proved by Deligne and Mumford in \cite{DM}.
Moreover, if $\mathcal{C}$
is a stable model for $C$
over $\mathcal{O}_F,$
and $F \subset E \subset \mathbf{C},$
then $\mathcal{C} \otimes_{\mathcal{O}_F} \mathcal{O}_E$
is a stable model for $C \otimes_F E$ over
$\mathcal{O}_E.$
The special fiber of any stable model for $C$
is called the stable reduction.

The stable models of $X_0(p)$ and
$X_0(p^2)$ were previously known,
due to works of J-I.\ Igusa and Deligne-Rapoport \cite[Section 7.6]{DR},
 and B.\ Edixhoven \cite[Theorem 2.1.2]{E} or \cite{E2} respectively. 
In \cite{CM}, R.\ Coleman and K.\ McMurdy calculated 
the stable reduction of $X_0(p^3),$
using the notion of stable coverings
of a rigid-analytic curve by basic wide opens.
In loc.\ cit., they use the Woods Hole theory in \cite{WH} 
and the Gross-Hopkins theory in \cite{GH}
 to deduce the defining equations of all irreducible components in the stable 
 reduction of $X_0(p^3).$ 
 They also determine 
 the stable reduction of $X_0(Np^3)$ with $(N,p)=1$
and compute the inertia action on the stable model
 of $X_0(p^3)$ in \cite{CM2}. 
See \cite[Introduction]{CM}
for other prior results regarding the stable models
of modular curves at prime power levels.
In \cite{T}, we compute the stable reduction of the modular curve $X_0(p^4)$
on the basis of the Coleman-McMurdy's work.
Similarly as \cite{CM}, we actually construct a stable covering
of $X_0(p^4)$ by basic wide opens.
To compute the reduction of irreducible components in
$X_0(p^4),$ we use the Kronecker polynomial.

Let $\mathcal{X}(\pi^n)$
denote the Lubin-Tate space.
This space is a rigid-analytic deformation space
with Drinfeld $\pi^n$-level structure of a one-dimensional formal 
$\mathcal{O}_F$-module
of height $h$
over the residue field $\mathbb{F}_q$
of $F.$
Using the type theory of Bushnell-Kutzko and Deligne-Carayol's 
non-abelian Lubin-Tate theory
 for ${\rm GL}_2$ in \cite{W3}[Theorem 3.3], 
J.\ Weinstein determines the stable model of $\mathcal{X}({\pi}^n)\ (h=2),$ 
{\it up to purely inseparable map} in \cite{W3}[Theorem 1.1].
More precisely, he proves that all irreducible components
admit a purely inseparable map to one of the following
four curves;
\\1.\ the projective line $\mathbb{P}^1$
\\2.\ the curve with Artin-Schreier affine model $a^q-a=s^2$
\\3.\ Deligne-Lusztig curve with affine model 
$x^qy-xy^q=1$
\\4.\  the curve with affine model $a^q-a=t^{q+1}.$

In this paper, we calculate precisely defining equations
of all irreducible components in
the stable reduction of the Lubin-Tate space $\mathcal{X}(\pi^2)\ ({\rm char}\ F=p>0)$ 
 on the basis of the ideas in \cite{CM} and in \cite{T}.
 Similarly as \cite{T}, our method in this paper
 is purely local, very explicit and elementary with using {\it blow-up.}
Techniques and ideas in loc.\ cit.\ can be applied
to the Lubin-Tate space $\mathcal{X}(\pi^2).$
The $\pi$-multiplication 
of the universal formal 
$\mathcal{O}_F$-module $\mathcal{F}^{\rm univ}$ over an open unit ball $\mathcal{X}(1)$ 
has the following simple form, if we choose 
an isomorphism
$\mathcal{X}(1) \simeq B(1) \ni u$ appropriately,
$$[\pi]_{\mathcal{F}^{\rm univ}}(X)=X^{q^2}+uX^q+\pi X.$$
This fact is crucial for our computation.
By using this equation, we can directly compute the reduction of all irreducible
components in the stable reduction of the Lubin-Tate
space $\mathcal{X}(\pi^2).$
We define several rigid analytic subspaces of $\mathcal{X}(\pi^2),$
which we denote by $\mathbf{Y}_{2,2},\mathbf{Y}_{3,1}$ and $\mathbf{Z}_{1,1}.$
We compute the reduction of these spaces. 
Irreducible components in the stable reduction of the Lubin-Tate space
$\mathcal{X}(\pi^2)$
consist of the reduction of these spaces.
We briefly introduce definitions of these spaces.
See subsection 3.1 for the precise definition. See notation below
for the rigid analytic notations.
Let $p_2:\mathcal{X}(\pi^2) \longrightarrow
\mathcal{X}(1)$ be the natural forgetful map.
We set as follows
$$\mathbf{Y}_{2,2}:=p_2^{-1}(B[p^{-\frac{q}{q+1}}]),\mathbf{Y}_{3,1}:=p_2^{-1}(C[p^{-\frac{1}{q+1}}]),
\mathbf{Z}_{1,1}:=p_2^{-1}(C[p^{-\frac{1}{2}}]).$$

Let $\mathcal{X}_{\rm LT}(\pi^n)$
be the zero-dimensional Lubin-Tate space 
and ${\rm LT}$ the universal formal group of height
$1$ over $\mathcal{X}_{\rm LT}(1).$ Let $F_0$
be the completion of the maximal unramified extension of $F$
in a fixed algebraic closure.
Furthermore, let $F_n=F_0({\rm LT}[\pi^n])$
denote the classical Lubin-Tate extension. Then, it is well-known that
the base change $\mathcal{X}(\pi^n)\times_F F_n$
has $q^{n-1}(q-1)$ connected components.
Hence, we write
$\mathcal{X}(\pi^n)\times_{F}
F_n=
\coprod_{\pi_n \in \mathcal{X}_{\rm LT}(\pi^n)(F_n)}\mathcal{X}^{\pi_n}(\pi^n).$
The connected component $\mathcal{X}^{\pi_n}(\pi^n)$
is defined by the following equation
$$\mu_n(X_n,Y_n)=\pi_n$$
where $\mu_n$ is called the Moore
determinant. See subsection 2.4 for more detail.
We set $\mathbf{Y}^{\pi_n}_{a,b}:=\mathbf{Y}_{a,b} \cap \mathcal{X}^{\pi_n}(\pi^n)$
and
$\mathbf{Z}^{\pi_n}_{a,b}:=\mathbf{Z}_{a,b} \cap \mathcal{X}^{\pi_n}(\pi^n).$

In the following, we explain defining equations of the reduction of these spaces
$\mathbf{Y}^{\pi_2}_{2,2}, \mathbf{Y}^{\pi_2}_{3,1}$ and $\mathbf{Z}^{\pi_2}_{1,1}.$
The reduction of $\mathbf{Y}^{\pi_2}_{2,2}$
is defined by the following equations
$$x^qy-xy^q=1,Z^q=x^{q^3}y-xy^{q^3}.$$
This affine curve has $q(q^2-1)$ singular points at
$(x,y)$ with
$x=\zeta y,\zeta \in \mathbb{F}^{\times}_{q^2} \backslash \mathbb{F}^{\times}_q$
and $y^{q+1}=\frac{1}{\zeta^q-\zeta}.$
We analyze the residue classes of these singular points.
Then, by blowing up the singular points, we find 
$q(q^2-1)$ irreducible components defined by $a^q-a=t^{q+1}.$
Similar phenomenon is observed 
in the stable reduction of the modular curve $X_0(p^4)$ in \cite[Section 4]{T}.

The reduction of the space $\mathbf{Y}^{\pi_2}_{3,1}$ has
$(q+1)$ connected components and each component is defined by $x^qy-xy^q=1,w^{q^2}=y.$
Let $\overline{\mathbf{Y}}^{\pi_2}_{3,1,\zeta}$
denote a connected component of $\overline{\mathbf{Y}}^{\pi_2}_{3,1}.$

The reduction of $\mathbf{Z}^{\pi_2}_{1,1}$
has $(q+1)$ connected components and each component 
is defined by the following
equation
$$Z^q=X^{q^2-1}+\frac{1}{X^{q^2-1}}$$
with genus $0.$
Let $\overline{\mathbf{Z}}^{\pi_2}_{1,1,\zeta}$
denote a connected component of the reduction $\overline{\mathbf{Z}}^{\pi_2}_{1,1}.$
This affine curve has $2(q^2-1)$
singular points at $X=\zeta,\zeta \in \mu_{2(q^2-1)}.$
By analyzing the residue classes of these singular points,
we find
$2(q^2-1)$ irreducible components defined by $a^q-a=s^2.$
Similar phenomenon is observed 
in the stable reduction of the modular curve $X_0(p^3)$ in \cite{CM}.
We also compute the inertia action 
and ${\rm GL}_2$-action
on the reduction of 
the spaces $\mathbf{Y}^{\pi_2}_{2,2},\mathbf{Y}^{\pi_2}_{3,1}$ and $\mathbf{Z}^{\pi_2}_{1,1}$
explicitly.
These computations of the defining equations of all irreducible components
in the stable reduction of $\mathcal{X}(\pi^2)$
are done in Section 4.

In section 5, we analyze the whole spaces $\mathcal{X}(\pi)$
and $\mathcal{X}(\pi^2).$
More precisely, we prove that
the complements $\mathcal{X}(\pi) \backslash \mathbf{Y}_{1,1}$ and
$\mathcal{X}(\pi^2) \backslash (\mathbf{Z}_{1,1} \cup \mathbf{Y}_{2,2} \cup 
\mathbf{Y}_{3,1})$
are disjoint unions of annuli.
Hence, we conclude that the wide open space $\mathcal{X}(\pi)$
is a basic wide open space.
On the other hand, the space $\mathcal{X}(\pi^2)$
is {\it not} basic wide open.
In subsection 5.3, we actually construct a stable covering of the wide open space
$\mathcal{X}(\pi^2)$ on the basis of the idea of Coleman-McMurdy 
\cite[Section 9]{CM}.
Furthermore, we give intersection multiplicities
in the stable reduction of the Lubin-Tate space $\mathcal{X}(\pi^2)$
in subsection 5.4.

We explain a shape of the stable reduction of the Lubin-Tate space 
$\mathcal{X}(\pi^2).$
Let $\overline{\mathbf{Y}}^c_{2,2}$ 
be the projective completion of the affine curve
 $\overline{\mathbf{Y}}^{\pi_2}_{2,2}.$
 Then, the complement 
 $\overline{\mathbf{Y}}^c_{2,2} \backslash 
 \overline{\mathbf{Y}}^{\pi_2}_{2,2}$
 consist of $(q+1)$ closed points.
The projective curve $\overline{\mathbf{Y}}^c_{2,2}$
meets the projective completion 
$\{\overline{\mathbf{Z}}^c_{1,1,\zeta}\}$
of $(q+1)$ 
affine curves $\{\overline{\mathbf{Z}}^{\pi_2}_{1,1,\zeta}\}$
at each infinity.
The complement $\overline{\mathbf{Z}}^c_{1,1,\zeta}
\backslash \overline{\mathbf{Z}}^{\pi_2}_{1,1,\zeta}$
consists of two closed points.
The projective curve
$\overline{\mathbf{Z}}^c_{1,1,\zeta}$
meets the projective completion 
$\overline{\mathbf{Y}}^c_{3,1,\zeta}$
of $(q+1)$ affine curves $\overline{\mathbf{Y}}_{3,1,\zeta}$
at each infinity.
The curve $\overline{\mathbf{Y}}^c_{3,1,\zeta}$
meets the Igusa curve ${\rm Ig}(p^2)$ at each infinity.
Since the affine curve $\overline{\mathbf{Y}}^{\pi_2}_{3,1,\zeta}$
has $(q+1)$ infinity points,
there exist $q(q+1)$ Igusa curves ${\rm Ig}(p^2)$
in the stable reduction of the Lubin-Tate space $\mathcal{X}^{\pi_2}(\pi^2).$
 
We are greatly inspired by the works of Coleman-McCallum in \cite{CW} on the stable reduction
of the quotient of the Fermat curve and of Coleman-McMurdy on the stable reduction of the modular
 curve $X_0(p^3)$ in \cite{CM}.
 We are also inspired by the work of T.\ Yoshida \cite{Y} and J.\ Weinstein \cite{W}-\cite{W3}. 
We would like to thank Professor T.\ Saito  and A.\ Abbes for helpful comments
on this work and encouragements.
We would like to thank S.\ Yasuda  and
S.\ Kondo for their interest on our work and stimulating
discussions.
\begin{notation}
Let $\pi$ be a uniformizer of $F.$ 
We fix some $\pi$-adic notation.
We let $\mathbf{C}$
be the completion of a fixed algebraic closure of $F,$
with integer ring $\mathbf{R}$
and with $\mathfrak{m}_{\mathbf{R}}$ 
the maximal ideal of $\mathbf{R}.$
Let $v$ denote the unique valuation
on $\mathbf{C}$ with $v(p)=1,$
$|\cdot|$
the absolute value given by $|x|=p^{-v(x)}$
and $\mathcal{R}=|\mathbb{C}_p^\ast|=p^{\mathbb{Q}}.$
Throughout the paper, we let $F$
be a complete subfield of $\mathbf{C}$
with ring of integers $R_F$ and residue field $\mathbb{F}_F.$
For $r \in \mathcal{R},$
we let $B_F[r]$ and $B_F(r)$
denote the closed and open disk over $F$
of radius $r$
around $0,$ i.e.\
the rigid spaces over $F$
whose $\mathbf{C}$-valued points
are $\{x \in \mathbf{C}:|x| \leq r\}$ and $\{x \in \mathbf{C}:|x|<r\}$
respectively.
If $r, s \in \mathcal{R}$ and $r \leq s,$
let $A_F[r,s]$ and $A_F(r,s)$ be the rigid spaces over $F$
whose $\mathbf{C}$-valued points are 
$\{x \in \mathbf{C}:r\leq |x| \leq s\}$ and 
$\{x \in \mathbf{C}:r<|x|<s\},$
which we call  closed annuli and open annuli.
By the width of such an annulus,
we mean ${\rm log}_p(s/r).$
A closed annuli of width $0$
will be called a circle, which we will also denote the circle, $A_F[s,s],$
by $C_F[s].$

\end{notation}
\section{Preliminaries\ (\cite{W2} and \cite{CM2})}
\subsection{definition of formal modules}
We begin with the definitions of formal $\mathcal{O}_F$-modules.
\begin{definition}
Let $R$ be a commutative $\mathcal{O}_F$-algebra, with structure map $i:\mathcal{O}_F
\longrightarrow R.$
A formal one-dimensional $\mathcal{O}_F$-module $\mathcal{F}$
is a power series $\mathcal{F}(X,Y)=X+Y+\cdots \in R[[X,Y]]$
which is commmutative, associative, admits $0$ as an identitiy,
together with a power series $[a]_{\mathcal{F}}(X) \in R[[X]]$
for each $a \in \mathcal{O}_F$ satisfying $[a]_{\mathcal{F}}(X) \equiv i(a) X\ {\rm mod}\ X^2$
and $\mathcal{F}([a]_{\mathcal{F}}(X),[a]_{\mathcal{F}}(Y))=[a]_{\mathcal{F}}(\mathcal{F}(X,Y)).$
\end{definition}

The addition law on a formal $\mathcal{O}_F$-module $\mathcal{F}$
will usually be written $X+_{\mathcal{F}}Y.$
If $R$ is a $k$-algebra, we either have $[\pi]_{\mathcal{F}}(X)=0$
or else $[\pi]_{\mathcal{F}}(X)=f(X^{q^h})$
for some power series $f(X)$
with $f'(0) \neq 0.$
In the latter case, we say $\mathcal{F}$
has height $h$ over $R.$
Let $\Sigma$
be a one-dimensional formal $\mathcal{O}_F$-module
over $\bar{k}$ of height $h.$
The functor of deformations of $\Sigma$ 
to complete local Noetherian $\hat{\mathcal{O}}^{\rm ur}_F$-algebra
is reresentable by a universal deformation
$\mathcal{F}^{\rm univ}$
over an algebra $\mathcal{A}$
which is isomorphic to the power series ring $\hat{\mathcal{O}}^{\rm ur}_F[[u_1,..,u_{h-1}]]$
in $(h-1)$ variables, cf \cite{Dr}.
That is , if $A$ is a complete local 
$\hat{\mathcal{O}}^{\rm ur}_F$-algebra with maximal ideal $P,$
then, the isomorphism classes of deformations of $\Sigma$ to $A$
are given exactly by specializing each $u_i$
to an element of $P$ in $\mathcal{F}^{\rm univ}.$

\subsection{The universal deformation in the equal characteristic case}

Assume ${\rm char}F=p>0,$ so that $F=k((\pi))$
is the field of Laurent series over $k$
in one variable,
with $\mathcal{O}_F=k[[\pi]].$
then, a model for $\Sigma$
is given by the simple rules
$$X+_{\Sigma}Y=X+Y,[\zeta]_{\Sigma}(X)=\zeta X, \zeta \in k, [\pi]_{\Sigma}(X)=X^{q^h}.$$

The universal deformation of $\mathcal{F}^{\rm univ}$
also has a simple model over $\mathcal{A} \simeq 
\hat{\mathcal{O}}^{\rm ur}_F[[u_1,..,u_{h-1}]]:$
$$X+_{\mathcal{F}^{\rm univ}}Y=X+Y$$
$$[\zeta]_{\mathcal{F}^{\rm univ}}(X)=\zeta X, \zeta \in k$$
\begin{equation}\label{rp1}
[\pi]_{\mathcal{F}^{\rm univ}}(X)=\pi X+u_1 X^q+\cdots+u_{h-1}X^{q^{h-1}}+X^{q^h}.
\end{equation}

\subsection{Moduli of deformations with level structure}
Let $A$ be a complete local $\mathcal{O}_F$
with maximal ideal $M,$ and let $\mathcal{F}$
be a one-dimensional $\mathcal{O}_F$-module over $A,$ and let $h>1$
be the height of $\mathcal{F} \otimes A/M.$
\begin{definition}
Let $n>1.$
A Drinfeld level $\pi^n$-structure
on $\mathcal{F}$ is an $\mathcal{O}_F$-module homomorphism
$$\phi:(\pi^{-n}\mathcal{O}_F/\mathcal{O}_F)^h \longrightarrow M$$
for which
the relation
$$\prod_{z \in (\pi^{-n}\mathcal{O}_F/\mathcal{O}_F)^h}(X-\phi(x))\ |\ [\pi]_{\mathcal{F}}(X)$$
holds in $A[[X]].$
If $\phi$ is a Drinfeld level $\pi^n$-structure, the image of $\phi$
of the stabdard basis
elements $(\pi^{-n},0,..,0),\dots,(0,0,..,\pi^{-n})$
of $(\pi^{-n}\mathcal{O}_F/\mathcal{O}_F)^h$
form a Drinfeld basis of $\mathcal{F}[\pi^n].$
\end{definition}

Fix a formal $\mathcal{O}_F$-module $\Sigma$
of height $h$ over $k.$
Let $A$ be a noetherian local $\hat{\mathcal{O}}^{\rm ur}_F$-algebra
such that the structure morphism $\hat{\mathcal{O}}^{\rm ur}_F \longrightarrow A$
induces an isomorphism between residue fields.
A deformation of $\Sigma$
with level $\pi^n$-structure 
over $A$ is a triple $(\mathcal{F},\eta,\phi)$
where $\eta:\mathcal{F} \otimes k \simeq \Sigma$
is an isomorphims of $\mathcal{O}_F$-modules over $k$
and $\phi$ is a Drinfeld level $\pi^n$-structure
on $F.$
\begin{proposition}(\cite{Dr})
The functor which assigns to each $A$
as above the set of deformations of $\Sigma$
with Drinfeld level $\pi^n$-structure over $A$
is represented by a regular local ring $\mathcal{A}(\pi^n)$
of dimension $h-1$ over $\hat{\mathcal{O}}^{\rm ur}_F.$
Let $X_1^{(n)},..X_{h}^{(n)}$
be the corresponding
Drinfeld basis for $\mathcal{F}^{\rm univ}[\pi^n].$
Then, these elements form a set of regular parameters for $\mathcal{A}(\pi^n).$
\end{proposition}

There is a finite injection of $\hat{\mathcal{O}}^{\rm ur}_F$-algebras $[\pi]_u:\mathcal{A}(\pi^n)
\hookrightarrow \mathcal{A}(\pi^{n+1})$
corresponding to the obvious degeneration map of functors.
We therefore may consider $\mathcal{A}(\pi^{n})$
as a subalgebra of $\mathcal{A}(\pi^{n+1}),$
with the equation
$[\pi]_u(X_i^{(n)})=X_i^{(n+1)}$
holding in $\mathcal{A}(\pi^{n+1}).$
Let $X(\pi^n)={\rm Spf}\mathcal{A}(\pi^n),$
so that $X(\pi^n) \longrightarrow {\rm Spf}\hat{\mathcal{O}}^{\rm ur}_F$
is formally smooth of relative dimension
$h-1.$ Let $\mathcal{X}(\pi^n)$
be the generic fiber of $X(\pi^n);$
then $\mathcal{X}(\pi^n)$
is a rigid analytic variety.
The coordinates $X_i^{(n)}$
are then analytic functions on $\mathcal{X}(\pi^n)$
with values in the open unit disc.
We have that $\mathcal{X}(1)$
is the rigid analytic open unit polydisc of dimension $h-1.$
The group 
${\rm GL}_h(\mathcal{O}_F/\pi^n\mathcal{O}_F)$
acts on the right on $\mathcal{X}(\pi^n)$
and on the left on $A(\pi^n).$
The degeneration map $\mathcal{X}(\pi^n) \longrightarrow \mathcal{X}(1)$
is Galois with group ${\rm GL}_h(\mathcal{O}_F/\pi^n\mathcal{O}_F).$
For an element $M \in {\rm GL}_h(\mathcal{O}_F/\pi^n\mathcal{O}_F)$
and an analytic function $f$ on $\mathcal{X}(\pi^n),$
we write $M(f)$
for the translated function $z \mapsto f(zM).$
When $f$ happens to be one of the parameters $X_i^{(n)},$
there is a natural definition of $M(X_i^{(n)})$
when $M \in M_h(\mathcal{O}_F/\pi^n\mathcal{O}_F)$
is an arbitrary matrix: if $M=(a_{ij}),$ then 
$$M(X_i^{(n)})=[a_{j1}]_{\mathcal{F}^{\rm univ}}(X_1^{(n)})+
_{\mathcal{F}^{\rm univ}} \dots +_{\mathcal{F}^{\rm univ}}[a_{jh}]_{\mathcal{F}^{\rm univ}}(X_h^{(n)}).$$

\subsection{Determinants}
First, we briefly recall the determinant
of level $\pi$-structures  restricted to the case $h=2$
from \cite[Section 3]{W2}. 
Define the polynomial 
in $2$ variables
$$\mu(X_1,Y_1)=X_1^qY_1-X_1Y_1^q \in k[X_1,Y_1].$$
This polynomial is $k$-linear alternating form, known as the Moore determinant.
Secondly, we recall a determinant of structure of higher level again restricted to the case $h=2$
from \cite[Section 3.3]{W2}.
Now let $n \geq 1,$
and suppose $X_n,Y_n$
are sections of $\mathcal{F}^{\rm univ}[\pi^n].$
We simply write $[\pi^a]_u(X)$ for $[\pi^a]_{\mathcal{F}^{\rm univ}}(X).$
We define the form $\mu_n$
$$\mu_n(X_n,Y_n)=\sum_{(a_1,a_2)}\mu([\pi^{a_1}]_u(X_n),[\pi^{a_2}](Y_n)),$$
where the sum runs over pairs of integers $(a_1,a_2)$
with $1 \leq a_i \leq n$
whose sum is $n.$
This is $k$-alternating in $X_n,Y_n.$
It is proved that $\mu_n$ is $\mathcal{O}_F$-linear
in \cite[Proposition 3.7]{W2}.

Let ${\rm LT}$
be a one-dimensional formal $\mathcal{O}_F$-module
over $\hat{\mathcal{O}}^{\rm ur}_F$ for which ${\rm LT} \otimes \bar{k}$
has height one.
Let $F_0=\hat{F}^{\rm ur},$ and for $n \geq 1,$
let $F_n=F_0({\rm LT}[\pi^n])$
be the classical Lubin-Tate extension.
Finally, let $\mathcal{X}_{\rm LT}(\pi^n)$
be the zero-dimensional space of deformations of ${\rm LT} \otimes \bar{k}$
with Drinfeld $\pi^n$-structure, so that $\mathcal{X}_{\rm LT}(\pi^n)(F_n)$
is the set of bases for ${\rm LT}[\pi^n](F_n)$
as a free $(\mathcal{O}_F/\pi^n\mathcal{O}_F)$-module
of rank one.

For the remainder of the paper, ${\rm LT}$
will denote the formal $\mathcal{O}_F$-module over $\hat{\mathcal{O}}^{\rm ur}_F$
with operations
$$X+_{\rm LT}Y=X+Y$$ 
$$[\alpha]_{\rm LT}(X)=\alpha X, \alpha \in k$$
$$[\pi]_{\rm LT}(X)=\pi X+(-1)^{h-1}X^q.$$
We introduce the following theorem proved in \cite[Theorem 3.2]{W2}.

\begin{theorem}(\cite[Theorem 3.2]{W2})
Assume ${\rm char}\ F=p>0.$
For each $n \geq 1,$
there exists
a morphism
$$\mu_n:\mathcal{F}^{\rm univ}[\pi^n]^{\times h} \longrightarrow {\rm LT}[\pi^n] \otimes \mathcal{A}$$
of group schemes over $\mathcal{A} \simeq \hat{\mathcal{O}}^{\rm ur}_F[[u_1,..u_{h-1}]]$
which
is $\mathcal{O}_F$-multilinear
and alternating and 
which satisfis the following properties;
\\1.\ The maps $\mu_n$
are compatible in the sense that
$$[\pi]_{\rm LT}(\mu_n(X_1,..,X_{h}))=\mu_{n-1}([\pi]_u(X_1),..,[\pi]_u(X_h))$$
for $n \geq 2.$
\\2.\ If $X_1,..,X_h$ are 
sectoins of $\mathcal{F}^{\rm univ}[\pi^n]$
over an $\mathcal{A}$-algebra $R$ which form a Drinfeld level $\pi^n$-
structure, then 
$\mu_n(X_1,..,X_h)$ is a Drinfeld level $\pi$ structure for ${\rm LT}[\pi^n] \otimes \mathcal{A}.$
\end{theorem}
The base change
$\mathcal{X}(\pi^n)\times_{F}F_n$
has $q^{n-1}(q-1)$ connected components and write 
$\mathcal{X}(\pi^n)\times_{F}F_n=
\coprod_{\pi_n \in \mathcal{X}_{\rm LT}(\pi^n)(F_n)}\mathcal{X}^{\pi_n}(\pi^n).$
Then, each connected component
$\mathcal{X}^{\pi_n}(\pi^n)$
is defined by the equation
$$\mu_n(X_n,Y_n)=\pi_n$$
by the above theorem.
See \cite[subsection 3.6]{W3} and \cite[Theorem 4.4]{St}
for more detail on geometrically connected components of
the Lubin-Tate space $\mathcal{X}(\pi^n).$

\subsection{Action of Inertia}
We will recall the action of inertia on the stable model of a curve over $\mathbf{C}$
from \cite[Section 6]{CM2}.

If $Y/F$
is a curve, and $\mathcal{Y}$ its stable model over $\mathbf{C},$
there is a homomorphism
$w_Y$
$$w_Y:I_F:={\rm Aut}_{\rm cont}(\mathbf{C}/F^{\rm ur})
\longrightarrow {\rm Aut}(\overline{\mathcal{Y}}).$$
It is characterized by the fact that for each $P \in Y(\mathbf{C})$
and $\sigma \in I_F,$
\begin{equation}\label{ine1}
\overline{P^{\sigma}}=w_Y(\sigma)(\overline{P}).
\end{equation}
We have something similar if $\mathbf{Y}$
is a reduced affinoid over $F.$
Namely, we have a homomorphism
$w_Y:I_F \longrightarrow {\rm Aut}(\overline{\mathbf{Y}}_{\mathbf{C}})$
characterized by (\ref{ine1}).
This follows from the fact that $I_F$
preserves $(\mathbf{Y}_{\mathbf{C}})^{0}$
(power bounded elements of $A(\mathbf{Y}_{\mathbf{C}})$)
and $A(\mathbf{Y}_{\mathbf{C}})^v$
(topologically nilpotent elements of $A(\mathbf{Y}_{\mathbf{C}})$).
Moreover, inertia action behaves well with respect to morphisms
in the following sense.
\begin{lemma}(\cite[Lemma 6.1]{CM2})\label{ine2}
If $f:X \longrightarrow Y$ is a morphism of reduced affinoids over $F$
and $\sigma \in I_F,$ then
$w_Y(\sigma) \circ \bar{f}=\bar{f} \circ w_X(\sigma).$
\end{lemma}

\section{Several subspaces in $\mathcal{X}(\pi^n)$}
Throughout the remainder of the paper,
we fix the following notations and assumptions.
Let $F$ be a non-archimedean local field of equal characteristic
with residue field $\mathbb{F}_q.$
Let $\mathcal{X}(\pi^n)$
be the Lubin-Tate space over $F$ and $\mathcal{F}^{\rm univ}$
the universal formal $\mathcal{O}_F$-module over $\mathcal{X}(1).$
We assume $h=2.$ 
We fix an identification
 $\mathcal{A} \simeq 
\hat{\mathcal{O}}^{\rm ur}_F[[u]]$ or $\mathcal{X}(1) \simeq B(1) \ni u$
such that
$[\pi]_u(X):=[\pi]_{\mathcal{F}^{\rm univ}}(X)=X^{q^2}+uX^q+\pi X$
as in (\ref{rp1}).
The set of $\mathbf{C}$-valued points $\mathcal{X}(\pi^n)(\mathbf{C})$
is identified with the following
$$\{(u,X_n,Y_n) \in \mathbf{C}^{\times 3}|
v(u)>0,\mu_n(X_n,Y_n) \neq 0,[\pi^n]_u(X_n)=[\pi^n]_u(Y_n)=0\}.$$ 
Let $\pi_n \in \mathcal{X}_{\rm LT}(\pi^n)(F_n).$ 
Then,
the set of $\mathbf{C}$-valued points $\mathcal{X}^{\pi_n}(\pi^n)(\mathbf{C})$
is identified with the following
$$\{(u,X_n,Y_n) \in \mathbf{C}^{\times 3}|
v(u)>0,\mu_n(X_n,Y_n)=\pi_n,[\pi^n]_u(X_n)=[\pi^n]_u(Y_n)=0\}.$$ 
We write $[\pi^i]_u(X_n)=X_{n-i},[\pi^i]_u(Y_n)=Y_{n-i}$
for $0 \leq i \leq n-1.$

\subsection{Subspaces $\mathbf{Y}_{a,b}$ and $\mathbf{Z}_{a,b}$ in $\mathcal{X}(\pi^n)$}
In this subsection, 
we define several subspaces $\mathbf{Y}_{a,b}$ and $\mathbf{Z}_{a,b}$
of the Lubin-Tate space $\mathcal{X}(\pi^n).$
We expect that the reduction of these spaces plays a fundamental role in
the stable reduction of the 
{\it Drinfeld modular curve.}
Actually, the reduction of the spaces
$\mathbf{Y}_{3,1},\mathbf{Y}_{2,2}$ and $\mathbf{Z}_{1,1}$
becomes irreducible components of the Lubin-Tate space $\mathcal{X}(\pi^2).$
 
Let $n \geq 1$ be a positive integer.
Let $p_n:\mathcal{X}(\pi^n) \longrightarrow \mathcal{X}(1);(F, \eta, \phi) \mapsto (F,\eta)$
be the natural forgetful map.
Let $\mathbf{TS}^0$ be a closed disc $B[p^{-\frac{q}{q+1}}] \subset \mathcal{X}(1) \simeq B(1).$ 
This is called the ``too-supersingular locus.''
For $(F,\eta) \in \mathbf{TS}^0,$ it is known that
the formal group $F$  has no canonical subgroup.
We define a subspace $\mathbf{Y}_{2n-m,m} \subset \mathcal{X}(\pi^n)\ (1 \leq m \leq n)$ 
as follows;
$$\mathbf{Y}_{n,n}:=p_n^{-1}(\mathbf{TS}^0) \subset \mathcal{X}(\pi^n),$$
$$\mathbf{Y}_{2n-m,m}:=p_n^{-1}(C[p^{-\frac{1}{q^{n-m-1}(q+1)}}]) \subset \mathcal{X}(\pi^n)\ 
(1 \leq m \leq n-1).$$

Let $n \geq 2$ be a positive integer.
For $1 \leq m \leq n-1,$ we define subspaces $\mathbf{Z}_{2(n-1)-m,m} \subset 
\mathcal{X}(\pi^n)$
as follows
$$\mathbf{Z}_{2(n-1)-m,m}:=p_n^{-1}(C[p^{-\frac{1}{2q^{(n-1)-m}}}]) \subset 
\mathcal{X}(\pi^n)\ (1 \leq m \leq n-1).$$
For a subspace $\mathbf{X} \subset \mathcal{X}(\pi^n)$
and $\pi_n \in \mathcal{X}_{\rm LT}(\pi^n)(F_n),$
 we set 
$\mathbf{X}^{\pi_n}:=\mathbf{X} \cap \mathcal{X}^{\pi_n}(\pi^n).$

\subsection{Subspaces of the spaces $\mathbf{Y}_{3,1}$ and $\mathbf{Z}_{1,1}$}
To compute the reduction of the spaces 
$\mathbf{Y}_{3,1}$ and $\mathbf{Z}_{1,1},$
we decompose these spaces to
disjoint
unions of several subspaces of them.
Let $(u,X_2,Y_2) \in \mathcal{X}(\pi^2).$
We define several subspaces of $\mathbf{Y}_{3,1}$
and $\mathbf{Z}_{1,1}$ by conditions of valuations of parameters
 $X_i,Y_i\ (i=1,2).$
Recall that we have $v(u)=\frac{1}{q+1}$ and $v(u)=1/2$
on the spaces $\mathbf{Y}_{3,1}$ and $\mathbf{Z}_{1,1}$
respectively.
\begin{definition}\label{deff}
1.\ We define a subspace $(u,X_2,Y_2) \in \mathbf{Y}_{3,1,e_1} \subset \mathbf{Y}_{3,1}$
 by the following conditions;
$$v(X_1)=\frac{q}{q^2-1},v(X_2)=\frac{1}{q(q^2-1)},v(Y_1)=\frac{1}{q(q^2-1)},v(Y_2)=\frac{1}{q^3(q^2-1)}.$$
\\2.\ We define a subspace $(u,X_2,Y_2) \in \mathbf{Y}_{3,1,e_1^{\vee}} \subset \mathbf{Y}_{3,1}$
 by the following condition;
$(u,X_2,Y_2) \in \mathbf{Y}_{3,1,e_1^{\vee}}$
is equivalent to $(u,Y_2,X_2) \in \mathbf{Y}_{3,1,e_1}.$
\\3.\ We define a subspace $(u,X_2,Y_2) \in \mathbf{Y}_{3,1,c} \subset \mathbf{Y}_{3,1}$
 by the following conditions;
$$v(X_1)=v(Y_1)=\frac{1}{q(q^2-1)},v(X_2)=v(Y_2)=\frac{1}{q^3(q^2-1)}.$$
4.\ We define a subspace $(u,X_2,Y_2) \in \mathbf{Z}_{1,1,e_1} \subset \mathbf{Z}_{1,1}$
 by the following conditions;
$$v(X_1)=\frac{1}{2(q-1)},
v(X_2)=\frac{1}{2q^2(q-1)},v(Y_1)=\frac{1}{2q(q-1)},v(Y_2)=\frac{1}{2q^3(q-1)}.$$
\\5.\ We define a subspace 
$(u,X_2,Y_2) \in \mathbf{Z}_{1,1,e_1^{\vee}} \subset \mathbf{Z}_{1,1}$
 by the following condition;
$(u,X_2,Y_2) \in \mathbf{Z}_{1,1,e_1^{\vee}}$
is equivalent to $(u,Y_2,X_2) \in \mathbf{Z}_{1,1,e_1}.$
\\6.\ We define a subspace $(u,X_2,Y_2) \in 
\mathbf{Z}_{1,1,c} \subset \mathbf{Z}_{1,1}$
 by the following conditions;
$$v(X_1)=v(Y_1)=\frac{1}{2q(q-1)},v(X_2)=v(Y_2)=\frac{1}{2q^3(q-1)}.$$
\end{definition}
\begin{lemma}
1.\ The space $\mathbf{Y}_{3,1}$
has the following description
$$\mathbf{Y}_{3,1}=\mathbf{Y}_{3,1,e_1} \coprod \mathbf{Y}_{3,1,e_1^{\vee}} \coprod
\mathbf{Y}_{3,1,c}.$$
\\2.\ The space $\mathbf{Z}_{1,1}$
has the following description
$$\mathbf{Z}_{1,1}=\mathbf{Z}_{1,1,e_1} \coprod \mathbf{Z}_{1,1,e_1^{\vee}} \coprod
\mathbf{Z}_{1,1,c}.$$
\end{lemma}
\section{Reduction of the spaces $\mathbf{Y}_{a,b}\ (a,b \geq 1, a+b=2,4)$ and $\mathbf{Z}_{1,1}$}
In this section,
we compute the reduction of 
the spaces 
$\mathbf{Y}_{1,1} \subset \mathcal{X}(\pi)$ and
$\mathbf{Y}_{3,1},\mathbf{Y}_{2,2}, \mathbf{Z}_{1,1} \subset \mathcal{X}(\pi^2)$
by only using blow-up.
Irreducible components 
of the stable reduction of $\mathcal{X}(\pi^2)$
consist of the reduction of these spaces.
See also Introduction for the defining
equations of the reduction of these spaces.
We also determine the inertia action on the reduction of these spaces.
 Furthermore, we also describe 
 the ${\rm GL}_2$-action on the stable reduction of $\mathcal{X}(\pi^2).$
Similar computation is also found in \cite[Sections 3 and 4]{T}.
\subsection{Calculation of the reduction of the space $\mathbf{Y}_{1,1} \subset \mathcal{X}(\pi)$}
We compute the reduction of the space $\mathbf{Y}_{1,1}.$
The reduction of the space is the Deligne-Lusztig curve as in the lemma below.
It is well-known that the Deligne-Lusztig curve with affine model $X^qY-XY^q=1$
appears in the stable reduction of the (Drinfeld) modular curve $X(\pi).$
This fact is also deduced from the Katz-Mazur model in \cite{KM}.
In this subsection, for the convenience of a reader, we write down a computation
of this component.
See also \cite[Proposition 6.15]{Y}, \cite{W2} and \cite[Thoerem 3.9]{W3}. 

Let $\pi_1 \in \mathcal{X}_{\rm LT}(\pi)(F_1).$
Then, we have $v(\pi_1)=1/(q-1).$
Let $(u,X_1,Y_1) \in X(\pi).$
Recall that the space $\mathbf{Y}_{1,1}$
is defined by the following conditions;
$v(u) \geq \frac{q}{q+1},v(X_1)=v(Y_1)=\frac{1}{q^2-1}.$
We choose an element $\alpha$ such that $\alpha^{q+1}=\pi_1.$
Then, we have $v(\alpha)=\frac{1}{q^2-1}.$
We consider an equation of $\mathbf{Y}^{\pi_1}_{1,1}$
\begin{equation}\label{ii1}
\mu(X_1,Y_1)=X_1^qY_1-X_1Y_1^q=\pi_1.
\end{equation}
We change variables as follows
$X_1=\alpha x,Y_1=\alpha y.$
Substituting them to the above equality (\ref{ii1})
 and dividing it by $\pi_1,$ we acquire the following 
 $$x^qy-xy^q=1.$$
 Therefore, we obtain the following lemma.
 \begin{lemma}
 The reduction of the space $\mathbf{Y}^{\pi_1}_{1,1}$
 is defined by the following equation
 $$x^qy-xy^q=1.$$
 This curve is called the Deligne-Luztig curve for ${\rm SL}_2(\mathbb{F}_q).$
 The genus of the curve is equal to $q(q-1)/2.$
 \end{lemma}
\begin{lemma}
Let $\sigma \in I_F$ be an element fixing $\pi_1.$
We write $\sigma(\alpha)=\zeta \alpha$ with $\zeta \in \mu_{(q+1)}.$
Then, the element $\sigma \in I_F$
acts on the reduction $\overline{\mathbf{Y}}^{\pi_1}_{1,1}$
as follows 
$$\sigma:\overline{\mathbf{Y}}^{\pi_1}_{1,1} \longrightarrow 
\overline{\mathbf{Y}}^{\pi_1}_{1,1};
(x,y) \mapsto (\zeta^{-1}x,\zeta^{-1}y).$$  
\end{lemma}
\begin{lemma}
Let $g=\left(
\begin{array}{cc}
a &  b\\
c & d
\end{array}
\right)
\in {\rm SL}_2(\mathcal{O}_F/\pi \mathcal{O}_F).
$
Then, the element $g$ acts on the reduction of $\mathbf{Y}^{\pi_1}_{1,1}$
as follows
$$g:\overline{\mathbf{Y}}^{\pi_1}_{1,1} \longrightarrow 
\overline{\mathbf{Y}}^{\pi_1}_{1,1};
(x,y) \mapsto (\bar{a}x+\bar{c}y,\bar{b}x+\bar{d}y).$$
\end{lemma}

\subsection{Computation of the reduction 
of the space $\mathbf{Y}_{3,1,e_1} \subset \mathcal{X}(\pi^2)$}
In this subsection, we compute the reduction of the space $\mathbf{Y}_{3,1,e_1}.$
We prove that the reduction $\overline{\mathbf{Y}}^{\pi_2}_{3,1,e_1}$
is defined by the Deligne-Lusztig equation
$-x^qy^{q^2}+xy^{q^3}=1$
in lemma below.

If $v(f-g)>\alpha$ with $\alpha \in \mathbb{Q}_{\geq 0},$ 
we write $f \equiv g\ ({\rm mod}\ \alpha+).$
Let $(u,X_2,Y_2) \in \mathbf{Y}_{3,1,e_1}.$
First, recall that the space $\mathbf{Y}_{3,1,e_1}$
is defined by the following conditions;
$$v(u)=\frac{1}{q+1}, v(X_1)=\frac{q}{q^2-1},
v(X_2)=\frac{1}{q(q^2-1)},v(Y_1)=\frac{1}{q(q^2-1)},
v(Y_2)=\frac{1}{q^3(q^2-1)}.$$
Let $\pi_2 \in \mathcal{X}_{\rm LT}(\pi^2)(F_2).$
We choose an element $\alpha$
such that $\alpha^{q^2(q+1)}=\pi_2$ with $v(\alpha)=1/q^3(q^2-1).$
We consider a defining equation of $\mathbf{Y}^{\pi_2}_{3,1,e_1}$
\begin{equation}\label{fd1}
\mu_2(X_2,Y_2)=X_1Y_2^q-X_1^qY_2-X_2^qY_1+X_2Y_1^q=\pi_2.
\end{equation}
We change variables as follows
$X_1=\alpha^{q^4}x_1,Y_1=\alpha^{q^2}y_1,X_2=\alpha^{q^2} x,Y_2=\alpha y.$
Substituting them to the equality (\ref{fd1}) and dividing it by $\pi_1,$
we acquire the following
$$-x^qy_1+xy_1^q=1.$$
Hence, this induces the following $-x^qy^{q^2}+xy^{q^3}=1,$
because we have $y_1=y^{q^2}$ by $[\pi]_u(Y_2)=Y_1.$
Therefore, we have obtained the following 
\begin{lemma}
The reduction of the space $\mathbf{Y}^{\pi_2}_{3,1,e_1}$
is defined by the following equation
$$-x^qy^{q^2}+xy^{q^3}=1.$$
\end{lemma}
\begin{lemma}
Let $\sigma \in I_F$ be an element fixing $\pi_2.$
We write $\sigma(\alpha)=\zeta \alpha$ with $\zeta \in \mu_{q^2(q+1)}.$
Then, the element $\sigma \in I_F$
acts on the reduction $\overline{\mathbf{Y}}^{\pi_2}_{3,1,e_1}$
as follows 
$$\sigma:\overline{\mathbf{Y}}^{\pi_2}_{3,1,e_1} 
\longrightarrow \overline{\mathbf{Y}}^{\pi_2}_{3,1,e_1};
(x,y) \mapsto (\zeta^{-q^2}x,\zeta^{-1}y).$$  
\end{lemma}
\begin{remark}
By the definition (\ref{deff}), the space $\mathbf{Y}^{\pi_2}_{3,1,e_1^{\vee}}$
has the same reduction as the one of the space $\mathbf{Y}^{\pi_2}_{3,1,e_1}$
by swapping $X_i$ for $Y_i\ (i=1,2).$
\end{remark}

\subsection{Calculation of the reduction of 
the space $\mathbf{Y}_{3,1,c} \subset \mathcal{X}(\pi^2)$}
In this subsection, we compute the reduction of the space $\mathbf{Y}^{\pi_2}_{3,1,c}.$
We prove that the reduction of the space $\mathbf{Y}^{\pi_2}_{3,1,c}$
has $(q-1)$ connected components and each component 
is defined by the same equation $-x^qy^{q^2}+xy^{q^3}=1$ 
as the one of the space $\overline{\mathbf{Y}}^{\pi_2}_{3,1,e_1}.$ 

Let $\pi_2$ be as in the previous subsection.
Let $(u,X_2,Y_2) \in \mathbf{Y}^{\pi_2}_{3,1,c}.$
Recall that the space $\mathbf{Y}^{\pi_2}_{3,1,c}$
is defined by the following conditions;
$$v(u)=\frac{1}{q+1},v(X_1)=v(Y_1)=\frac{1}{q(q^2-1)},v(X_2)=v(Y_2)=\frac{1}{q^3(q^2-1)}.$$
We choose an element $\alpha$ such that $\alpha^{q^2(q+1)}=\pi_2$ with $v(\alpha)=1/q^3(q^2-1).$
Then, we change variables as follows
$X_1=\alpha^{q^2}x_1,Y_1=\alpha^{q^2} y_1,X_2=\alpha x,Y_2=\alpha y.$
Substituting them to the equality (\ref{fd1}) and dividing it by $\alpha^{q(q+1)},$
we acquire the following equality
\begin{equation}\label{ng1}
(x_1y^q-x^qy_1)-\gamma(x_1^qy-xy_1^q)=\gamma^{q/(q-1)}
\end{equation}
where we set $\gamma:=\alpha^{(q-1)(q^2-1)}$ with $v(\gamma)=(q-1)/q^3.$
Since we have the following congruence
$x_1 \equiv x^{q^2},y_1 \equiv y^{q^2}$ modulo $(1/q^2)+,$
the equality (\ref{ng1}) induces the following congruence
\begin{equation}\label{ng2}
(x^qy-xy^q)^q-\gamma (x^{q^3}y-xy^{q^3}) \equiv \gamma^{q/(q-1)}\ ({\rm mod}\ (1/q^2)+).
\end{equation}
We change variables as follows $a=x/y,t=1/y$ with $a$ an invertible function.
Furthemore, we set $\mathcal{Z}:=\frac{a^q-a}{t^{q+1}}.$
Then, the above congruence (\ref{ng2})
has the following form
\begin{equation}\label{ng3}
\mathcal{Z}^q-\gamma \biggl(s\mathcal{Z}^{q^2}+
\frac{\mathcal{Z}^q}{s^{(q-1)}}+\frac{\mathcal{Z}}{s^{q}}\biggr)
\equiv \gamma^{q/(q-1)}\ ({\rm mod}\ (1/q^2)+)
\end{equation}
where we set $s:=t^{q^2-1}.$
By this congruence, we have
$v(\mathcal{Z})=1/q^3.$
Now, we change a variable as follows
$\mathcal{Z}=\gamma^{1/(q-1)}z.$
Substituting this to 
(\ref{ng3}), and dividing this by $\gamma^{q/(q-1)},$
we acquire the following congruence
\begin{equation}\label{ng4}
z^q-\frac{z}{s^{q}}=1\ ({\rm mod}\ 0+).
\end{equation}
We change variables as follows
$x:=-zt^q,y=1/t.$
Then, the equation (\ref{ng4})
has the following form
$$-x^qy^{q^2}+xy^{q^3}=1.$$
Note that $a$ is an invertible function.
By $\mathcal{Z} \equiv 0\ ({\rm mod}\ 0+),$ we acquire $a \in \mathbb{F}^{\times}_q.$
Hence, we have obtained the following
\begin{proposition}
The reduction of the space $\mathbf{Y}^{\pi_2}_{3,1,c}$
is a disjoint union of $(q-1)$ curves defined by the following equation
 $$-x^qy^{q^2}+xy^{q^3}=1.$$
\end{proposition}
Let $\{\overline{\mathbf{Y}}^{\pi_2}_{3,1,c,\zeta}\}_{\zeta \in \mathbb{F}^{\times}_q}$
be connected components of the reduction of $\mathbf{Y}^{\pi_2}_{3,1,c}.$
We compute the inertia action on the reduction of the space $\mathbf{Y}^{\pi_2}_{3,1,c}.$
\begin{lemma}
Let $\sigma \in I_F$ be an element fixing $\pi_2.$
We write $\sigma(\alpha)=\zeta \alpha$ with $\zeta \in \mu_{q^2(q+1)}.$
Then, the element $\sigma \in I_F$
acts on the reduction $\overline{\mathbf{Y}}^{\pi_2}_{3,1,c,\zeta}$
as follows 
$$\sigma:\overline{\mathbf{Y}}^{\pi_2}_{3,1,c,\zeta} 
\longrightarrow \overline{\mathbf{Y}}^{\pi_2}_{3,1,c,\zeta};
(X,Y) \mapsto (\zeta^qX,\zeta^{-1}Y).$$  
\end{lemma}
\begin{proof}
Note that $x^{\sigma}=\zeta^{-1}x,y^{\sigma}=\zeta^{-1}y.$
Hence, we acquire $a^{\sigma}=a,t^{\sigma}=\zeta t$ and $z^{\sigma}=z.$
Therefore, the required assertion follows.
\end{proof}
We describe the action of ${\rm SL}_2(\mathcal{O}_F/\pi^2\mathcal{O}_F)$
on the reduction of the space $\mathbf{Y}^{\pi_2}_{3,1,e_1}.$
\begin{lemma}
Let $g=\left(
\begin{array}{cc}
a & b \\
c & d
\end{array}
\right)
 \in {\rm SL}_2(\mathcal{O}_F/\pi^2\mathcal{O}_F).
$
\\1.\ If $c,d$ are units, $\overline{\mathbf{Y}}^{\pi_2}_{3,1,e_1}$
goes to $\overline{\mathbf{Y}}^{\pi_2}_{3,1,c,\bigl(
\frac{\bar{c}}{\bar{d}}\bigr)}$ 
by the action of $g.$
Moreover, $g$ acts as follows;
$$g:\overline{\mathbf{Y}}^{\pi_2}_{3,1,e_1}
\longrightarrow
\overline{\mathbf{Y}}^{\pi_2}_{3,1,c,\bigl(\frac{\bar{c}}{\bar{d}}\bigr)};
(x,y) \mapsto (\frac{x}{\bar{d}},\bar{d}y).$$
\\2.\ If $c$ is a unit and $d$ is divisible by $\pi,$
$\overline{\mathbf{Y}}^{\pi_2}_{3,1,e_1}$
goes to $\overline{\mathbf{Y}}^{\pi_2}_{3,1,e_1^{\vee}}$ 
by the action of $g.$
Further, the element $g$ acts as follows;
$$g:\overline{\mathbf{Y}}^{\pi_2}_{3,1,e_1}
\longrightarrow
\overline{\mathbf{Y}}^{\pi_2}_{3,1,e_1^{\vee}};
(x,y) \mapsto (\bar{c}y,\bar{b}x+\overline{\biggl(\frac{d}{\pi}\biggr)}y^{q^2}).$$
\\3.\ If $c$ is divisible by $\pi$ and $d$ is a unit,
$\overline{\mathbf{Y}}_{3,1,e_1}$
is stable under the action of $g.$
Further, $g$ acts as follows;
$$g:\overline{\mathbf{Y}}^{\pi_2}_{3,1,e_1} \longrightarrow 
\overline{\mathbf{Y}}^{\pi_2}_{3,1,e_1};
(x,y) \mapsto 
(\bar{a}x+\overline{\biggl(\frac{c}{\pi}\biggr)}y^{q^2},\bar{d}y).$$
\end{lemma}

\subsection{Calculation of the reduction of the 
space $\mathbf{Y}_{2,2} \subset \mathcal{X}(\pi^2)$}
In this subsection, we compute the reduction of the space $\mathbf{Y}_{2,2}.$
We prove that
the reduction of the space $\mathbf{Y}^{\pi_2}_{2,2}$
is defined by the following equations;
$$x^qy-xy^q=1,Z^q=x^{q^3}y-xy^{q^3}.$$
This affine curve has $q(q^2-1)$ singular points
at $(x,y)$ with $x=\zeta y, \zeta \in \mathbb{F}^{\times}_{q^2} 
\backslash \mathbb{F}^{\times}_q$ and
$y^{q+1}=\frac{1}{\zeta^q-\zeta}.$

Let $\pi_2$ be as in the previous subsection.
Let $(u,X_2,Y_2) \in \mathbf{Y}_{2,2}.$
We recall that $\mathbf{Y}^{\pi_2}_{2,2}$ is defined by the following conditions;
$$v(u) \geq \frac{q}{q+1},v(X_1)=v(Y_1)=\frac{1}{q^2-1},v(X_2)=v(Y_2)=\frac{1}{q^2(q^2-1)}.$$

We choose an element $\alpha$
such that $\alpha^{q(q+1)}=\pi_2.$

Now, we consider the equality (\ref{fd1}).
Then, we change variables
as follows
$X_1=\alpha^{p^2}x_1,X_2=\alpha x_2,Y_1=\alpha^{p^2} y_1,Y_2=\alpha y_1.$
Substituting them to the equality (\ref{fd1}) and dividing it by $\pi_2,$
we acquire the following
\begin{equation}\label{tome} 
(x_1y_2^q-x_2^qy_1)-\gamma (x_1^qy_2-y_1^qx_2)=1\ ({\rm mod}\ (1/q)+).
\end{equation}
where we set $\gamma:=\alpha^{(q-1)(q^2-1)}.$
Since we have $y_1 \equiv y_2^{q^2}\ ({\rm mod}\ (1/q)+),x_1 \equiv x_2^{q^2}\ ({\rm mod}\ (1/q)+),$
(\ref{tome}) induces the following congruence
\begin{equation}\label{f2}
(x_2^qy_2-x_2y_2^q)^q-\gamma (x_2^{q^3}y_2-y_2^{q^3}x_2)=1\ ({\rm mod}\ (1/q)+).
\end{equation}
In the following, we simply write $x,y$ for $x_2,y_2.$
Now,
we introduce a new parameter $Z$ as follows
$x^qy-xy^q=1+\gamma_1Z$
where the element $\gamma_1$ satisfies $\gamma_1^q=\gamma.$
Substituting $x^qy-xy^q=1+\gamma_1Z$
 to the above congruence (\ref{f2}), and dividing it by $\gamma,$
 we obtain the following congruence
 \begin{equation}
 Z^q \equiv x^{q^3}y-xy^{q^3}\ ({\rm mod}\ (1/q^2)+).
 \end{equation}
Hence, we obtain the following proposition.

\begin{proposition}\label{pro1}
The reduction of the space $\mathbf{Y}^{\pi_2}_{2,2}$
is defined by the following equations
$$x^qy-xy^q=1,Z^q=x^{q^3}y-xy^{q^3}.$$  
In particular, this curve is an affine curve of genus $q(q-1)/2.$
This curve has singular points at $x=\zeta y$ with $\zeta 
\in \mathbb{F}^{\times}_{q^2} \backslash \mathbb{F}^{\times}_{q}$
and $y^{q+1}=\frac{1}{\zeta^q-\zeta}.$
Hence, this curve has $q(q^2-1)$ singular points.
\end{proposition}
\begin{lemma}
Let $\sigma \in I_F$ be an element fixing $\pi_2.$
Let $\alpha_1$ be an element such that $\alpha_1^q=\alpha.$
We write $\sigma(\alpha_1)=\zeta \alpha_1$ with $\zeta \in \mu_{q^2(q+1)}.$
Then, the element $\sigma \in I_F$
acts on the reduction $\overline{\mathbf{Y}}^{\pi_2}_{2,2}$
as follows 
$$\sigma:\overline{\mathbf{Y}}^{\pi_2}_{2,2} 
\longrightarrow \overline{\mathbf{Y}}^{\pi_2}_{2,2};
(x,y,Z) \mapsto (\zeta^{-q}x,\zeta^{-q}y,Z).$$  
\end{lemma}
\begin{lemma}
Let $g=\left(
\begin{array}{cc}
a &  b\\
c & d
\end{array}
\right)
\in {\rm SL}_2(\mathcal{O}_F/\pi^2\mathcal{O}_F).
$
For an element $a \in \mathcal{O}_F/\pi^2\mathcal{O}_F,$
we denote by $\bar{a}$
the image of $a$ by the canonical
map $\mathcal{O}_F/\pi^2\mathcal{O}_F \longrightarrow
\mathcal{O}_F/\pi \mathcal{O}_F.$ 
Then, the element $g$ acts on the reduction of $\mathbf{Y}^{\pi_2}_{2,2}$
as follows
$$g:\overline{\mathbf{Y}}^{\pi_2}_{2,2} 
\longrightarrow \overline{\mathbf{Y}}^{\pi_2}_{2,2};
(x,y,Z) \mapsto (\bar{a}x+\bar{c}y,\bar{b}x+\bar{d}y,Z).$$
\end{lemma}

\subsection{Analysis of singular residue classes in $\mathbf{Y}_{2,2}$}
In this subsection, we analyze the singular residue classes of the space 
$\mathbf{Y}^{\pi_2}_{2,2}.$
We find $q(q^2-1)$ irreducible components defined by $a^q-a=t^{q+1}$
which attach to the curve in Proposition \ref{pro1} at each
singular point.
In \cite{T}, we prove that, for each supersingular point,
 there exist $(p+1)$ components
defined by $a^p-a=t^{p+1}$
in the stable reduction of the modular curve
$X_0(p^4).$
A computation in this subsection is very similar to
the one
in \cite[subsection 4.4]{T}.

We keep the same notation as in the previous subsection.
We change variables as follows
$a:=x/y,t:=1/y.$
Then, we have the following congruences by the computations in the previous
subsection
\begin{equation}\label{t^1}
a^q-a=t^{q+1}(1+\gamma_1Z)
\end{equation}
\begin{equation}\label{t^2}
Z^q \equiv \frac{a^{q^3}-a}{t^{q^3+1}}\ ({\rm mod}\ (1/q^2)+).
\end{equation}
We set $s:=t^{q^2-1}.$
Then, the congruence (\ref{t^2})
has the following form under the variables $(Z,s)$
\begin{equation}\label{t^3}
Z^q \equiv \frac{(s+1)^q}{s^{q-1}}+\frac{1+\gamma_1Z}{s^q}\ ({\rm mod}\ (1/q^2)+).
\end{equation}
We set $s+1:=s_1$ and consider a locus $v(s_1)=1/q^2(q+1).$
We can easily check that
the following congruence holds on the term in the right hand side of the congruence
(\ref{t^3})
$$
\frac{(s+1)^q}{s^{q-1}}+
\frac{1+\gamma_1Z}{s^q} \equiv -1-\gamma_1 Z-s_1^{q+1}\ ({\rm mod}\ (1/q^2)+).$$
Hence, (\ref{t^2}) is written as follows
\begin{equation}\label{t^4}
Z^q \equiv -1-\gamma_1 Z-s_1^{q+1}\ ({\rm mod}\ (1/q^2)+).
\end{equation}
We choose an element $\gamma_0$
such that $\gamma_0^q+1+\gamma_1\gamma_0=0.$
Further, we choose elements $\beta,\beta_1$
such that $\beta^{q-1}=-\gamma_1,\beta_1^{q+1}=\beta^q.$
Then, we have $v(\beta)=1/q^3$ and $v(\beta_1)=1/q^2(q+1).$

We change variables as follows
$$Z=\gamma_0+\beta a, s_1=\beta_1 s_2.$$
Substituting them to the congruence (\ref{t^4})
and dividing it by $\beta^q,$ we acquire the following by the definitions of $\gamma_0,\beta,\beta_1$
$$a^q-a=s_2^{q+1}\ ({\rm mod}\ 0+).$$
Hence, we have proved the following proposition.
\begin{proposition}
In the reduction of the space $\mathbf{Y}^{\pi_2}_{2,2},$
there exist $q(q^2-1)$ irreducible components defined by $a^q-a=t^{q+1},$
which attach to the curve in Proposition \ref{pro1} at each singular point.
\end{proposition}
Let $\{\mathcal{D}_{\zeta}\}$
be the underlying affinoid of the singular 
residue classes of $\mathbf{Y}^{\pi_2}_{2,2}.$
\begin{lemma}
Let $\sigma \in I_F$ be an element fixing $\pi_2.$
Then, the element $\sigma \in I_F$
acts on the reduction $\overline{\mathcal{D}}_{\zeta}$
as follows 
$$\sigma:\overline{\mathcal{D}}_{\zeta} \longrightarrow \overline{\mathcal{D}}_{\zeta};
(a,s_2) \mapsto 
(\overline{\biggl(\frac{\gamma_0-\sigma(\gamma_0)}
{\sigma(\beta)}\biggr)}+\overline{
\biggl(\frac{\beta}{\sigma(\beta)}\biggr)}a,
\overline{\biggl(\frac{\beta_1}{\sigma(\beta_1)}\biggr)}s_2)$$
with 
$\overline{\bigl(\frac{\gamma_0-\sigma(\gamma_0)}
{\sigma(\beta)}\bigr)} \in \mathbb{F}_q,$
$\overline{
\bigl(\frac{\beta}{\sigma(\beta)}\bigr)} \in \mathbb{F}^{\times}_q$
and
$\overline{\bigl(\frac{\beta_1}{\sigma(\beta_1)}\bigr)}^{q+1}=\overline{
\bigl(\frac{\beta}{\sigma(\beta)}\bigr)}.$
\end{lemma}

\subsection{Calculation of the reduction of the space $\mathbf{Z}_{1,1,e_1} \subset \mathcal{X}(\pi^2)$}
We will compute the reduction of the space $\mathbf{Z}_{1,1,e_1}.$
We prove that
the reduction of the space $\mathbf{Z}^{\pi_2}_{1,1,e_1}$
is defined by the following equation
$$Z^q=X^{q^2-1}+\frac{1}{X^{q^2-1}}.$$
This affine curve with genus $0$
has $2(q^2-1)$ singular points at $X=\zeta,\zeta \in \mu_{2(q^2-1)}.$
Similar phenomenon is already observed in 
the defining equation of ``bridging component'' 
$\overline{\mathbf{Z}}^A_{1,1}$ in the stable reduction of the modular curve
$X_0(p^3)$ found by Coleman-McMurdy in \cite{CM}[Proposition 8.2].
See also \cite{T}[Proposition 3.1].

Let $\pi_2$ be as in the previous subsection.
Let $(u,X_2,Y_2) \in \mathbf{Z}_{1,1,e_1}.$
Now, we recall that the space $\mathbf{Z}_{1,1,e_1}$
is defined by the following conditions;
$$v(u)=1/2,
v(X_1)=\frac{1}{2(q-1)},v(X_2)=\frac{1}{2q^2(q-1)},
v(Y_1)=\frac{1}{2q(q-1)},v(Y_2)=\frac{1}{2q^3(q-1)}.$$
We choose an element $\alpha$ such that 
$\alpha^{2q^2}=\pi_2.$
Then, we have $v(\alpha)=1/2q^3(q-1).$

We change variables as follows
$X_1=\alpha^{q^3}x_1,X_2=\alpha^{q}x,Y_1=\alpha^{q^2} y_1, Y_2=\alpha y.$
Substituting them to (\ref{fd1}) and dividing it by $\pi_2,$
we acquire the following
congruence
\begin{equation}\label{ff1}
-x^qy_1+\gamma (x_1y^q+xy_1^q)=1\ ({\rm mod}\ (1/2q)+)
\end{equation}
where we set $\gamma:=\alpha^{q(q-1)^2}.$
We have $v(\gamma)=(q-1)/2q^2.$
Since we have $y_1 \equiv y_2^{q^2},x_1 \equiv x_2^{q^2}$
modulo $(1/2q)+,$
the congruence (\ref{ff1}) induces the following congruence
\begin{equation}\label{ff2}
-x^qy^{q^2}+\gamma(x^{q^2}y^q+xy^{q^3})=1\ ({\rm mod}\ (1/2q)+).
\end{equation}
Then, we introduce a new parameter $Z$
as follows
$1+xy^{q}=\gamma_1 Z$ where $\gamma_1$ satisfies $\gamma_1^q=\gamma.$
By substituting $1+xy^{q}=\gamma_1 Z$ to (\ref{ff2})
and dividing it by $\gamma,$ we acquire the following congruence
\begin{equation}\label{ff4}
\biggl(Z+\frac{1}{y^{q^2-1}}+y^{q^2-1}\biggr)^q
 \equiv \gamma_1 y^{q(q^2-1)}Z\ ({\rm mod}\ (1/2q^2)+).
\end{equation}
Again, we introduce a new parameter $Z_1$ as follows
\begin{equation}\label{ff3}
Z+\frac{1}{y^{q^2-1}}+y^{q^2-1}=\gamma_2 y^{q^2-1} Z_1
\end{equation}
where we choose an element $\gamma_2$ such that $\gamma_2^q=\gamma_1.$
Substituting (\ref{ff3})
 to (\ref{ff4}) and dividing it by $y^{q(q^2-1)}\gamma_1,$
 the following congruence holds
 $Z_1^q=Z\ ({\rm mod}\ (1/2q^3)+).$
 Furthermore, by substituting this to (\ref{ff3}), we obtain the following congruence
 \begin{equation}\label{ff5}
 Z_1^q+\frac{1}{y^{q^2-1}}+y^{q^2-1} \equiv \gamma_2 y^{q^2-1} Z_1\ ({\rm mod}\ (1/2q^3)+).
 \end{equation}
Hence, we have obtained the following proposition.

\begin{proposition}
The reduction of the space $\mathbf{Z}^{\pi_2}_{1,1,e_1}$
is defined by the following equation
$$Z_1^q+\frac{1}{y^{q^2-1}}+y^{q^2-1}=0.$$
In particular, this curve is an affine curve of genus $0.$
This curve has singular points at $y=\zeta$ with $\zeta \in \mu_{2(q^2-1)}.$
\end{proposition}
\begin{remark}
By the definition (\ref{deff}), the space $\mathbf{Z}^{\pi_2}_{1,1,e_1^{\vee}}$
has the same reduction as the one of the space $\mathbf{Z}^{\pi_2}_{1,1,e_1}$
by swapping $X_i$ for $Y_i\ (i=1,2).$
\end{remark}
\begin{lemma}
Let $\sigma \in I_F$ be an element fixing $\pi_2.$
We write $\sigma(\alpha)=\zeta \alpha$ with $\zeta \in \mu_{2q^2}.$
Then, the element $\sigma \in I_F$
acts on the reduction $\overline{\mathbf{Z}}^{\pi_2}_{1,1,e_1}$
as follows 
$$\sigma:\overline{\mathbf{Z}}^{\pi_2}_{1,1,e_1} 
\longrightarrow \overline{\mathbf{Z}}^{\pi_2}_{1,1,e_1};
(y,Z) \mapsto (\zeta^{-1}y,Z).$$  
\end{lemma}

\subsection{Analysis of the singular residue classes of 
$\mathbf{Z}^{\pi_2}_{1,1,e_1}$}
In this subsection, we analyze the singular residue 
classes in $\mathbf{Z}^{\pi_2}_{1,1,e_1}.$
Then, we find $2(q^2-1)$ irreducible components defined by 
the Artin-Schreier equation $a^q-a=t^2.$
A calculation in this subsection is very similar to
the one in \cite{T}[subsection 3.2].

We keep the same notation as in the previous subsection.
We set $s:=y^{q^2-1}.$
We recall the following congruence (\ref{ff5})
\begin{equation}\label{ff5}
 Z_1^q+\frac{1}{s}+\biggl(\frac{s}{1-\gamma_2Z_1}\biggr) \equiv 0\ ({\rm mod}\ (1/2q^3)+).
 \end{equation}
We set $F(s,Z_1):=Z_1^q+\frac{1}{s}+\bigl(\frac{s}{1-\gamma_2Z_1}\bigr).$

We choose an elements $\gamma_0$
such that
$\gamma_0^q+2(1+\gamma_2\gamma_0)^{1/2}=0.$
We set $s_0:=(1+\gamma_2\gamma_0)^{1/2}.$
Then, we have $\partial_s F(s_0,\gamma_0)=0, F(s_0,\gamma_0)=0.$
Furthermore, we choose elements $\beta,\beta_1$
such that $\beta^{q-1}=\gamma_2s_0, \beta_1^2=-\beta^q/s_0^3.$
Note that $v(\beta)=1/2q^4,v(\beta_1)=1/4q^3.$

We change variables as follows
$$Z_1=\gamma_0+\beta a, s=s_0+\beta_1s_1.$$
By substituting these to (\ref{ff5}), and dividing it by $\beta^q,$
we acquire the following, $a^q-a \equiv s_1^2\ ({\rm mod}\ 0+).$
Hence, we have obtained the following proposition.
\begin{proposition}
In the reduction of the space $\mathbf{Z}^{\pi_2}_{1,1,e_1},$
there exist $2(q^2-1)$ irreducible components defined by $a^q-a=t^2.$
\end{proposition}
Let $\mathcal{D}_{\zeta}$
be the underlying affinoid
of the singular residue classes in $\mathbf{Z}^{\pi_2}_{1,1,e_1}.$
\begin{lemma}
Let $\sigma \in I_F$ be an element fixing $\pi_2.$
Then, the element $\sigma \in I_F$
acts on the reduction $\overline{\mathcal{D}}_{\zeta}$
as follows 
$$\overline{\mathcal{D}}_{\zeta} \longrightarrow \overline{\mathcal{D}}_{\zeta};
(a,s_1) \mapsto 
(\overline{\biggl(\frac{\gamma_0-\sigma(\gamma_0)}
{\sigma(\beta)}\biggr)}+\overline{
\biggl(\frac{\beta}{\sigma(\beta)}\biggr)}a,
\overline{\biggl(\frac{\beta_1}{\sigma(\beta_1)}\biggr)}s_1)$$
with 
$\overline{\bigl(\frac{\gamma_0-\sigma(\gamma_0)}
{\sigma(\beta)}\bigr)} \in \mathbb{F}_q,$
$\overline{
\bigl(\frac{\beta}{\sigma(\beta)}\bigr)} \in \mathbb{F}^{\times}_q$
and
$\overline{\bigl(\frac{\beta_1}{\sigma(\beta_1)}\bigr)}^{2}=\overline{
\bigl(\frac{\beta}{\sigma(\beta)}\bigr)}.$
\end{lemma}

\subsection{Calculation of the reduction of the 
space $\mathbf{Z}_{1,1,c} \subset \mathcal{X}(\pi^2)$}
In this subsection, we compute the reduction of the space $\mathbf{Z}_{1,1,c}.$
The reduction of the space $\mathbf{Z}^{\pi_2}_{1,1,c}$
has $(q-1)$ connected components and each component is defined by
$Z^q=X^{q^2-1}+(1/X^{q^2-1})$ as the reduction $\overline{\mathbf{Z}}^{\pi_2}_{1,1,e_1}.$

Let $\pi_2$ be as in the previous subsection.
Let $(u,X_2,Y_2) \in \mathbf{Z}^{\pi_2}_{1,1,c}.$
Recall that the space $\mathbf{Z}^{\pi_2}_{1,1,c}$ is defined by the following conditions;
$$v(u)=1/2,v(X_1)=v(Y_1)=\frac{1}{2q(q-1)},v(X_2)=v(Y_2)=\frac{1}{2q^3(q-1)}.$$
We choose an element $\alpha$ such that $\alpha^{2q^2}=\pi_2$
with $v(\alpha)=1/2q^3(q-1).$
We change variables as follows
$X_1=\alpha^{q^2}x_1,Y_1=\alpha^{q^2}y_1,X_2=\alpha x,Y_2=\alpha y.$
Substituting them to (\ref{ff1}), and dividing this by $\alpha^{q(q+1)},$
we acquire the following 
\begin{equation}\label{ffi1}
(x_1y^q-x^qy_1)-\gamma^{q+1}(x_1^qy-xy_1^q)=\gamma^{q/(q-1)}
\end{equation}
where we set $\gamma:=\alpha^{(q-1)^2}.$ Then, we have $v(\gamma)=(q-1)/2q^3.$
Since we have $x_1 \equiv x^{q^2},y_1 \equiv y^{q^2}$ modulo
$(q+1)/2q^2+,$
the equality (\ref{ffi1}) induces the following congruence
\begin{equation}\label{ffi2}
(x^qy-xy^q)^q-\gamma^{q+1}(x^{q^3}y-xy^{q^3})=\gamma^{q/(q-1)}\ 
({\rm mod}\ \biggl(\frac{q+1}{2q^2}\biggr)+).
\end{equation}
We put
\begin{equation}\label{ffi3}
\mathcal{Z}:=x^qy-xy^q
\end{equation}
Further, we set $a:=x/y,t:=1/y.$
Substituting (\ref{ffi3}) to (\ref{ffi2}), we acquire the following
\begin{equation}\label{ngi3}
\mathcal{Z}^q-\gamma^{q+1} \biggl(t^{q^2-1}\mathcal{Z}^{q^2}+
\frac{\mathcal{Z}^q}{t^{(q-1)(q^2-1)}}+\frac{\mathcal{Z}}{t^{q(q^2-1)}}\biggr)
\equiv \gamma^{q/(q-1)}\ ({\rm mod}\ \biggl(\frac{q+1}{2q^2}\biggr)+).
\end{equation}
We set $\mathcal{Z}=\gamma^{1/(q-1)}+\frac{\gamma^{q/(q-1)}}{t^{q^2-1}}+\mathcal{Z}_1.$
Substituting this to (\ref{ngi3}), we obtain the following congruence
\begin{equation}\label{ngii1}
\mathcal{Z}_1^q-\gamma^{q+1}\frac{\mathcal{Z}_1}{t^{q(q^2-1)}}
 \equiv \gamma^{(q^2+q-1)/(q-1)}
 \biggl(\frac{1}{t^{(q+1)(q^2-1)}}+
 \frac{1}{t^{(q-1)(q^2-1)}}\biggr)\ ({\rm mod}\ \biggl(\frac{q+1}{2q^2}\biggr)+).
\end{equation}
We choose an element $\gamma_1$
such that $\gamma_1^q=\gamma.$
We change a variable as follows
$\mathcal{Z}_1=\gamma_1^{(q^2+q-1)/(q-1)}\frac{Z}{t^{q^2-1}}.$
Substituting this to (\ref{ngii1}) and dividing it by 
$\gamma^{(q^2+q-1)/(q-1)},$ we acquire the following
\begin{equation}\label{boi1}
Z^q \equiv \frac{1}{t^{q^2-1}}+t^{q^2-1}+\gamma_1\frac{Z}{t^{q^2-1}}\ ({\rm mod}\ (1/2q^3)+).
\end{equation}
Hence, we have obtained the following 
\begin{proposition}
The reduction of the space $\mathbf{Z}^{\pi_2}_{1,1,c}$
is a disjoint union of $(q-1)$ curves defined by the following equation
$$Z^q=\frac{1}{t^{q^2-1}}+t^{q^2-1}.
$$
Furthermore, there exist $2(q^2-1)$ irreducible components
defined by $a^q-a=s^{2}$ which attach to the above curve
at each singular point.
\end{proposition}
\begin{proof}
The required assertion follows from
(\ref{boi1}) and the computation in the previous subsection.
\end{proof}
We describe the action of ${\rm SL}_2(\mathcal{O}_F/\pi^2\mathcal{O}_F)$
on the reduction of the space $\mathbf{Z}^{\pi_2}_{1,1,e_1}.$
\begin{lemma}
Let $g=
\left(
\begin{array}{cc}
a & b \\
c & d
\end{array}
\right) 
\in
{\rm SL}_2(\mathcal{O}_F/\pi^2\mathcal{O}_F).$
For an element $a\in \mathcal{O}_F/\pi^2\mathcal{O}_F,$ 
we denote by $\bar{a}$
by the image of $a$ by the canonical map
$\mathcal{O}_F/\pi^2\mathcal{O}_F \longrightarrow \mathcal{O}_F/\pi \mathcal{O}_F.$
\\1.\ If $c,d$ are units, $\overline{\mathbf{Z}}^{\pi_2}_{1,1,e_1}$
goes to $\overline{\mathbf{Z}}^{\pi_2}_{1,1,c}.$
The element $g$ acts as follows;
$$g:\overline{\mathbf{Z}}^{\pi_2}_{1,1,e_1} 
\longrightarrow 
\overline{\mathbf{Z}}^{\pi_2}_{1,1,c};
(Z,y) \mapsto
(Z,\bar{d}y)$$
\\2.\ If $c$ is a unit and $d$ is divisible by $\pi$, 
$\overline{\mathbf{Z}}^{\pi_2}_{1,1,e_1}$
goes to $\overline{\mathbf{Z}}^{\pi_2}_{1,1,e_1^{\vee}}.$
The element $g$ acts as follows;
$$g:\overline{\mathbf{Z}}^{\pi_2}_{1,1,e_1} 
\longrightarrow 
\overline{\mathbf{Z}}^{\pi_2}_{1,1,e_1^{\vee}};
(Z,y) \mapsto (Z,\bar{c}x).$$
\\3.\ If $c$ is divisible by $\pi$ and $d$ is a unit, 
$\overline{\mathbf{Z}}^{\pi_2}_{1,1,e_1}$
is stable by $g.$
The element $g$ acts as follows;
$$g:\overline{\mathbf{Z}}^{\pi_2}_{1,1,e_1} 
\longrightarrow 
\overline{\mathbf{Z}}^{\pi_2}_{1,1,e_1};
(Z,y) \mapsto (Z,\bar{d}y)$$
\end{lemma}

\section{Analysis of the spaces $\mathcal{X}(\pi)$ and $\mathcal{X}(\pi^2)$}
In this section, we analyze the spaces $\mathcal{X}(\pi)$ and $\mathcal{X}(\pi^2).$
In the stable reduction of the spaces 
$\mathcal{X}(\pi^{i})\ (i=1,2),$ nothing interesting 
happens except for
the reduction of $\mathbf{Y}_{1,1},\mathbf{Y}_{3,1},\mathbf{Y}_{2,2}$
and $\mathbf{Z}_{1,1}.$
More precisely, we prove that
the complements $\mathcal{X}(\pi) \backslash \mathbf{Y}_{1,1}$ and
$\mathcal{X}(\pi^2) \backslash (\mathbf{Z}_{1,1} \cup \mathbf{Y}_{2,2} \cup 
\mathbf{Y}_{3,1})$
are disjoint unions of annuli. We prove it in subsections 5.1 and 5.2.
In other words, we prove that the inverse image 
by $p_2:\mathcal{X}(\pi^2) \longrightarrow \mathcal{X}(1)$
of every circle 
$C[p^{-\alpha}] \subset \mathcal{X}(1)$ with $\alpha \in \bigl((0,\frac{q}{q+1}) \cap 
\mathbb{Q}\bigr) \backslash \{\frac{1}{q+1},\frac{1}{2}\}$ 
 is a disjoint union of circles. 
As a result of the analysis in subsection 5.2,
 we will construct a stable covering
of the wide open space $\mathcal{X}(\pi^2)$
by defining basic wide open subspaces in $\mathcal{X}(\pi^2)$
which contain the spaces $\mathbf{Y}_{3,1},\mathbf{Y}_{2,2}$
and $\mathbf{Z}_{1,1}$ in subsection 5.3.
In subsection 5.4, we give intersection multiplicities 
for stable reduction of the space $\mathcal{X}(\pi^2).$
The notion of wide open spaces and stable coverings
is due to R.\ Coleman,
for example see \cite[Section 2]{CM}.

\subsection{Analysis of the space $\mathcal{X}(\pi)$}
We analyze the space $\mathcal{X}(\pi).$
We prove that the complement $\mathcal{X}(\pi) \backslash \mathbf{Y}_{1,1}$
is a disjoint union of annuli.
Hence, we conclude that the space $\mathcal{X}(\pi)$
is a basic wide open space.
This fact is well-known, but we write down a calculation
for the convenience of a reader.

\begin{lemma}\label{pg}
We consider the following equality
$$[\pi]_u(X)=X^{q^2}+uX^q+\pi X=0.$$
Then, we have the following
\\1.\ If $v(X)>1/(q^2-1),$
we have $v(X)=(1-v(u))/(q-1)$ and $v(u)<q/(q+1).$
\\2.\ If $v(X)=1/(q^2-1),$ we have $v(u) \geq q/(q+1).$
\\3.\ If $v(X)<1/(q^2-1),$ we have $v(X)=v(u)/q(q-1)$ and $v(u)<q/(q+1).$
\end{lemma}
Let $X,Y$ be $\pi$-torsion points of the universal formal group $\mathcal{F}^{\rm univ}$.
\begin{lemma}\label{pg2}
Let $(u,X,Y) \in \mathcal{X}(\pi).$ 
A case 1 for $X$ and $1$ for $Y$ in Lemma \ref{pg}
does not occur.
\end{lemma}
\begin{proof}
We consider the equality
$X^qY-XY^q=\pi_1.$
The valuation of the left hand side is larger than $(q+1)(1-v(u))/(q-1).$
By $v(u)<q/(q+1),$ we acquire $(q+1)(1-v(u))/(q-1)>v(\pi_1)=1/(q-1).$
Hence, this case does not happen.
\end{proof}
Let $1 \leq a,b \leq 3$ be positive integers.
Let $\mathbf{W}_{a,b}$ denote a subspace of 
$(u,X,Y) \in \mathcal{X}(\pi)$ 
defined by the conditions
$a$ for $X$ in Lemma \ref{pg} and $b$ for $Y$ in Lemma \ref{pg}.
Note that $\mathbf{W}_{a,b}$ except for $(a,b)=(1,3),(3,1),(2,2),(3,3)$
is empty by Lemmas \ref{pg} and \ref{pg2}.
Furthemore, we note that the space $\mathbf{W}_{2,2}$
is equal to $\mathbf{Y}_{1,1}.$
Recall that, for a subspace $\mathbf{X}$
and $\pi_n \in \mathcal{X}^{\pi_n}(\pi^n)(F_n),$
we write $\mathbf{X}^{\pi_n}$ for the intersection 
$\mathbf{X} \cap \mathcal{X}^{\pi_n}(\pi^n).$
\begin{lemma}\label{annl1}
The spaces $\mathbf{W}^{\pi_1}_{1,3}$ and 
$\mathbf{W}^{\pi_1}_{3,1}$
are annuli of width $1/(q^2-1).$
\end{lemma}
\begin{proof}
We prove the assertion only for the space $\mathbf{W}_{1,3}.$
By $v(u)<q/(q+1),$ we acquire $v(Y)<v(X).$
We have $v(XY^q)=1/(q-1).$
Hence, we acquire the following $XY^q=\pi_1+{\rm higher}\ {\rm terms}.$
Here , the valuation of the higher terms is greater
 than $1/(q-1).$
 Therefore, the required assertion follows.
\end{proof}
\begin{lemma}\label{annl2}
The space $\mathbf{W}^{\pi_1}_{3,3}$
is a disjoint union of $(q-1)$ annuli
with width $1/(q^2-1).$
\end{lemma}
\begin{proof}
Consider the equality $X^qY-XY^q=\pi_1.$
We set $X=\zeta Y+Z$ with $\zeta \in \mathbb{F}^{\times}_q$
and $v(Y)<v(Z).$
Then, the equality $X^qY-XY^q=\pi_1$ has the following form
$-Y^qZ=\pi_1+{\rm higher}\ {\rm terms}$
with $v(Z)=\frac{1-v(u)}{q-1}>v(Y).$
Here the valuation of the higher terms is greater than $1/(q-1).$
Hence, $X$ and $Z$ are written with respect to $Y.$
Thereby, we have proved the required assertion.
\end{proof}
By Lemmas \ref{annl1} and \ref{annl2}, we know
that
the space $\mathcal{X}(\pi)$
is a basic wide open space.
\begin{corollary}
The complement $\mathcal{X}(\pi) \backslash \mathbf{Y}_{1,1}$
is a union of annuli $\mathbf{W}_{1,3} \cup \mathbf{W}_{3,1}
\cup \mathbf{W}_{3,3}.$
In particular, $\mathcal{X}(\pi) \backslash \mathbf{Y}_{1,1}$
is a disjoint union of annuli.
In other words, the space $\mathcal{X}(\pi)$
is a basic wide open space.
\end{corollary}
\begin{proof}
The required assertion follows from
Lemmas \ref{annl1} and \ref{annl2}.
\end{proof}

\subsection{Analysis of the space $\mathcal{X}(\pi^2)$}
In this subsection,
we analyze the space $\mathcal{X}(\pi^2).$
To do so,
we define several subspaces
of $\mathcal{X}(\pi^2)$
and prove that the subspaces
are merely disjoint union of annuli.
Finally, we prove that the complement
$\mathcal{X}(\pi^2) \backslash (\mathbf{Y}_{3,1} \cup \mathbf{Z}_{1,1} \cup \mathbf{Y}_{2,2})$
is a disjoint union of annuli.
Propositions \ref{rop1} and \ref{rop2}
play a key role to
show that a covering $\mathcal{C}(\pi^2),$
which will be constructed in the next subsection,
is actually a stable covering.

Let $(u,X_2,Y_2) \in \mathbf{W}_{1}$ denote a subspace defined by the following conditions;
$$0<v(u)<\frac{1}{q+1},v(X_1)=\frac{1-v(u)}{q-1},v(X_2)=\frac{1-qv(u)}{q(q-1)}, 
v(Y_1)=\frac{v(u)}{q(q-1)},v(Y_2)=\frac{v(u)}{q^3(q-1)}.$$
Let $(u,X_2,Y_2) \in \mathbf{W}_{2}$ denote a subspace defined by the following conditions;
$$0<v(u)<\frac{1}{q+1},v(X_1)=\frac{1-v(u)}{q-1},v(X_2)=\frac{v(u)}{q(q-1)}, 
v(Y_1)=\frac{v(u)}{q(q-1)},v(Y_2)=\frac{v(u)}{q^3(q-1)}.$$
Let $(u,X_2,Y_2) \in \mathbf{W}_3$ denote a subspace defined by the following conditions;
$$\frac{1}{q+1}<v(u)<\frac{q}{q+1},
v(X_1)=\frac{1-v(u)}{q-1},v(X_2)=\frac{1-v(u)}{q^2(q-1)},
v(Y_1)=\frac{v(u)}{q(q-1)}
,v(Y_2)=\frac{v(u)}{q^3(q-1)}.$$
For $1 \leq i \leq 3,$ let 
$\mathbf{W}^{\vee}_{i} \subset \mathcal{X}(\pi^2)$
be a subspace defined by the following condition;
$(u,X_2,Y_2) \in \mathbf{W}^{\vee}_i$ 
is equivalent to $(u,Y_2,X_2) \in \mathbf{W}_{i}.$ 
Let $\mathbf{U}_1$ (resp.\ $\mathbf{U}_2$ resp.\ $\mathbf{U}_3$)
be a subspace defined by the following conditions;
$$v(X_1)=v(Y_1)=\frac{v(u)}{q(q-1)},
v(X_2)=v(Y_2)=\frac{v(u)}{q^3(q-1)}$$
and $0<v(u)<\frac{1}{q+1}.$ 
(resp.\ $\frac{1}{q+1}<v(u)<\frac{1}{2},$ 
resp.\ $\frac{1}{2}<v(u)<\frac{q}{q+1}.$)
\begin{proposition}\label{rop1}
Let the notation be as above.
Then, we have the followings
\\1.\ The spaces $\mathbf{W}^{\pi_2}_1$ and $\mathbf{W}^{\pi_2,\vee}_{1}$ are annuli of width $1/q^3(q^2-1).$
\\2.\ The spaces $\mathbf{W}^{\pi_2}_2$ and $\mathbf{W}^{\pi_2,\vee}_2$ are disjoint unions
of $(q-1)$ annuli of width $1/q^3(q^2-1).$
\\3.\ The spaces $\mathbf{W}^{\pi_2}_3 \backslash \mathbf{Z}^{\pi_2}_{1,1,e_1}$ 
and $\mathbf{W}^{\pi_2,\vee}_3 \backslash \mathbf{Z}^{\pi_2}_{1,1,e_1^{\vee}}$
are disjoint unions of two annuli of width $1/2q^4(q+1)$.
\end{proposition}
\begin{proof}
We prove the assertion 1.
Note that $v(X_2Y_1^q)=1/q(q-1).$
We consider the equality (\ref{fd1}).
Then, we acquire the following $X_2Y_1^q=\pi_2+{\rm higher}\ {\rm terms}.$
Here, the valuation of the higher terms is greater than $1/q(q-1).$
Hence, the space $\mathbf{W}_1(\mathbf{C})$
is isomorphic to 
$$\{Y_2 \in \mathbf{C}|0<v(Y_2)<\frac{1}{q^3(q^2-1)}\}.$$
Therefore, the required assertion follows.

We prove the assertion 2.
Note that $v(X_2^qY_1)=v(X_2Y_1^q)<1/q(q-1).$
By the equality (\ref{fd1}), the following holds
$X_2=\zeta Y_1+{\rm higher}\ {\rm terms}$ with some $\zeta \in \mathbb{F}^{\times}_q.$
Here, the valuation of the higher terms is greater than $v(X_2).$
Therefore, the space $\mathbf{W}_{2}$ splits to $(q-1)$ components, and
the $\mathbf{C}$-valued point of each component is identified with
$$\{Y_2 \in \mathbf{C}|0<v(Y_2)<\frac{1}{q^3(q^2-1)}\}.$$
Hence, the required assertion follows.

We prove the assertion 3.
Consider the equality (\ref{fd1})
\begin{equation}\label{fd22}
X_1Y_2^q-X_1^qY_2-X_2^qY_1+X_2Y_1^q=\pi_2.
\end{equation}
Note that the valuation of the term
$X_2^qY_1$ is smallest among the terms in the left hand side of the 
equality 
(\ref{fd22}).
Note that
$v(X_2Y_1^q)<v(X_1Y_2^q)$ is equivalent to
$v(u)<1/2.$
Moreover, note that
$v(X_1^qY_2)>v({X_2}^{q}Y_2)-v(Y_2^{q(q^2-1)})$
is equivalent to
$v(u)<\frac{q}{q+1}.$

By (\ref{fd22}) and $[\pi]_u(X_2)=X_1,[\pi]_u(Y_2)=Y_1$, 
we acquire the following equality
\begin{equation}\label{fr1}
X_2^qY_2^{q^2}-X_2^{q^2}Y_2^q=-\pi_2+X_2Y_2^{q^3}+{\rm higher}\ {\rm terms} 
\end{equation}
where the valuation of the higher terms is greater
than
${\rm min}\{(v(X_2)/q)+v(Y_2^{q^2(q+1)}),v(X_2^qY_2^{q^3+1})\}-v(Y_2^q).$
Note that
$(v(X_2)/q)+v(Y_2^{q^2(q+1)})<qv(X_2)+(q^3+1)v(Y_2)$
is equivalent to 
$v(u)<1/2.$
We introduce a new parameter
$Z$ as follows
\begin{equation}\label{fgs1}
X_2Y_2^{q}-X_2^qY_2=-\pi_2^{1/q}+Y_2^{q^2}Z.
\end{equation}
Substituting this to (\ref{fr1}) and dividing it by $Y_2^{q^3},$
we acquire the following
\begin{equation}\label{fr2}
Z^q=X_2+{\rm higher}\ {\rm terms}
\end{equation}
where the valuation of the higher terms is greater than 
${\rm min}\{v(Y_2^{q^2}Z),v(Z^{q^2}Y_2)\}-v(Y_2^q).$
By substituting (\ref{fr2}) to (\ref{fgs1}),
the following equality 
holds
\begin{equation}\label{fr4}
(ZY_2)^q=-\pi_2^{1/q}+Y_2^{q^2}Z+Z^{q^2}Y_2+{\rm higher}\ {\rm terms}
\end{equation}
where the valuation of the higher terms is greater than 
${\rm min}\{v(Y_2^{q^2}Z),v(Z^{q^2}Y_2)\}.$
Note that $v(Y_2^{q^2}Z)>v(Z^{q^2}Y_2)$
is equivalent to 
$v(u)>1/2.$

Now, we consider a case $\frac{1}{q+1}<v(u)<\frac{1}{2}.$
Again, we introduce a new parameter $Z_1$ as follows
$$
ZY_2=-\pi_2^{1/q^2}+Y_2^qZ_1.
$$
Substituting this to (\ref{fr4}) and dividing it by $Y_2^{q^2},$
we obtain the following equality
$Z_1^q=Z+{\rm higher}\ {\rm terms}.$
Here, the valuation of the higher terms is greater than $v(Z).$
Hecnce, the parameters $Z,X_2,Y_2$ are written with respect to $Z_1.$
Note that $v(Z_1)=\frac{1-v(u)}{q^4(q-1)}.$
Hence, the required assertion follows.

Secondly, we consider a case $\frac{1}{2}<v(u)<\frac{q}{q+1}.$
We set $ZY_2=-\pi_2^{1/q^2}+Z^qZ_1.$
Then, substituting this to (\ref{fr4}) and divding it by $Z^{q^2},$
we acquire the following
equality 
$Z_1^q=Y_2+{\rm higher}\ {\rm terms}.$
Here the valuation of the higher terms is greater than $v(Y_2).$
Hence, the required assertion follows.
\end{proof}
\begin{proposition}\label{rop2}
Let the notation be as above.
Then, we have the followings
\\1.\ The space $\mathbf{U}^{\pi_2}_1$ 
is a disjoint union of $q(q-1)$ annuli of width $1/q^3(q^2-1).$ 
\\2.\ The spaces $\mathbf{U}^{\pi_2}_2$ and $\mathbf{U}^{\pi_2}_3$ are disjoint unions 
of $(q-1)$ annuli of width $1/2q^4(q+1)$.
\end{proposition}
\begin{proof}
We prove the assertion 1.
Recall that we have $0<v(u)<\frac{1}{q+1}$ on $\mathbf{U}_1.$
Consider the equality (\ref{fd1}). 
Since we have $[\pi]_u(X_2)=X_1,[\pi]_u(Y_2)=Y_1,$ we acquire the following
\begin{equation}\label{ool1}
(X_2^qY_2-X_2Y_2^q)^q-(X_2^{q^3}Y_2-X_2Y_2^{q^3})={\rm higher}\ {\rm terms}
\end{equation}
where the valuation of the higher terms
is greater than
$v(Y_2^{q^2(q+1)}).$
Then, we set $a:=X_2/Y_2$ with $v(a)=0$ and $z:=a^q-a.$
By substituting these to the equality (\ref{ool1}),
and dividing it by $Y_2^{q(q+1)},$
 we acquire the following
$$z^{q}-Y_2^{(q-1)(q^2-1)}z={\rm higher}\ {\rm terms}$$
where the valuation of the higher terms
is greater than
$v(Y_2^{q(q^2-1)}).$
Therefore, we find $v(z)=v(Y_2^{q^2-1})>0$ and this equation splits to $q$-equations.
By $v(z)>0$, we acquire $a \in \mathbb{F}^{\times}_q.$
Therefore, the required assertion follows.

We prove the assertion 2.
Recall that we have $\frac{1}{q+1}<v(u)<\frac{q}{q+1},$
$v(X_1)=v(Y_1)=\frac{v(u)}{q(q-1)},$ and 
$v(X_2)=v(Y_2)=\frac{v(u)}{q^3(q-1)}$
on $\mathbf{U}_2 \cup \mathbf{U}_3.$
We consider the equality (\ref{fd1}).
Since we have $[\pi]_u(X_2)=X_1,[\pi]_u(Y_2)=Y_1,$ we obtain the following
equality
\begin{equation}\label{rou1}
(X_2^qY_2-X_2Y_2^q)^q-(X_2^{q^3}Y_2-X_2Y_2^{q^3})=\pi_2+{\rm higher}\ {\rm terms}
\end{equation}
where the valuation of the higher terms is greater than
${\rm min}\{v(Y_2^{(q-1)(q^2-1)}\pi_2),v(Y_2^{(q+1)(q^2-1)}\pi_2^{1/q^2})\}.$
We set $a:=X_2/Y_2$ with $v(a)=0$ and $z:=a^q-a.$
Then, the equality (\ref{rou1})
induces the following 
\begin{equation}\label{rou2}
Y_2^{q(q+1)}z^q-Y_2^{q^3+1}(z^{q^2}+z^q+z)=\pi_2+{\rm higher}\ {\rm terms}
\end{equation}
where the valuation of the higher terms is greater than
${\rm min}\{v(Y_2^{(q-1)(q^2-1)}\pi_2),v(Y_2^{(q+1)(q^2-1)}\pi_2^{1/q^2})\}.$
Note that $v(z)=\frac{1}{q^2(q-1)}-(q+1)v(Y_2)>0.$
We set $Y_2^{q+1}z=\pi_2^{1/q}+Y_2^{q^2-1}\pi_2^{1/q^2}+z_1.$
By substituting this to (\ref{rou2})  
, the equality (\ref{rou2}) has the following form
\begin{equation}\label{rou3}
z_1^q-Y_2^{(q-1)(q^2-1)}\pi_2-Y_2^{(q+1)(q^2-1)}\pi_2^{1/q^2}
-Y_2^{q(q^2-1)}z_1=0+{\rm higher}\ {\rm terms}
\end{equation}
where the valuation of the higher terms is greater than
${\rm min}\{v(Y_2^{(q-1)(q^2-1)}\pi_2),v(Y_2^{(q+1)(q^2-1)}\pi_2^{1/q^2})\}.$
Note that $v(Y_2^{(q-1)(q^2-1)}\pi_2)<v(Y_2^{(q+1)(q^2-1)}\pi_2^{1/q^2})$
is equivalent to $1/2<v(u).$

First, we consider a case $1/(q+1)<v(u)<1/2.$
In this case, we have
$v(z_1^q)=v(Y_2^{(q+1)(q^2-1)}\pi_2^{1/q^2})$
by (\ref{rou3}).
By $v(u)>\frac{1}{q+1},$ we obtain 
$v(z_1^q)<v(Y_2^{q(q^2-1)}z_1).$
Then, By setting $Z:=\frac{\pi_2^{1/q^3}Y_2^{q^2+q-1}}{z_1},$
we acquire the following $Y_2=Z^q+{\rm higher}\ {\rm terms}$
by (\ref{rou3}).
Here, the valuation of the higher terms is greater than
$v(Y_2).$
 Further, we have $a \in \mathbb{F}^{\times}_q.$
Therefore, the required assertion follows.

Secondly, we consider the other case $1/2<v(u)<q/(q+1).$
In this case, we have
$v(z_1^q)=v(Y_2^{(q-1)(q^2-1)}\pi_2)$
by (\ref{rou3}).
By $v(u)>1/2,$
we acquire $v(z_1^q)<v(Y_2^{q(q^2-1)}z_1).$
By setting $Z:=\frac{\pi_2^{1/q}Y_2^{q^2+q-1}}{z_1},$
we obtain $Y_2=Z^q+{\rm higher}\ {\rm terms}$
by (\ref{rou3}). 
Here, the valuation of the higher terms
is greater than $v(Y_2).$ Furthermore, we have $a \in \mathbb{F}^{\times}_q.$
Hence, the required assertion follows.
\end{proof}
\begin{corollary}
The complement $\mathcal{X}(\pi^2) \backslash (\mathbf{Y}_{3,1} \cup \mathbf{Z}_{1,1} \cup
\mathbf{Y}_{2,2})$
is a disjoint union of annuli.
\end{corollary}
\begin{proof}
The required assetion follows from
Propositions \ref{rop1} and \ref{rop2}.
\end{proof}

\subsection{Stable covering of $\mathcal{X}(\pi^2)$}
In this subsection,
we construct the stable covering of the wide open
space $\mathcal{X}(\pi^2).$
The space $\mathcal{X}(\pi^2)$
is not basic wide open.
We construct a covering
$\{\mathbf{V}_i\}_{i \in I}$
 of $\mathcal{X}(\pi^2)$ with
 the piece
 $\mathbf{V}_{i}$
 a basic wide open space.
 Namely, we prove that all intersections $\mathbf{V}_i \cap \mathbf{V}_j \neq \phi$
 are annuli.
 We prove that, if $i,j,k \in I$ are different from each other,
 the intersection
 $\mathbf{V}_i \cap \mathbf{V}_j \cap \mathbf{V}_k$
 is empty.
 Furthermore, we show that 
 the space $\mathbf{V}_i$ contains an underlying affinoid $Z_{\mathbf{V}_i}$ 
 and the complement
 $\mathbf{V}_i \backslash Z_{\mathbf{V}_i}$
 is equal to a disjoint union of annuli
 $\bigcup_{j \neq i}(\mathbf{V}_i \cap \mathbf{V}_j).$
Similar constructions of stable coverings
of modular curves
are found in [Section 9]\cite{CM} and \cite[subsection 5.2]{T}.
See \cite[Section 2]{CM} or \cite[section 2.3]{W3}
 for (basic) wide open spaces and stable coverings.

We define several subspaces of $\mathcal{X}(\pi^2).$
Let $(u,X_2,Y_2) \in \mathcal{X}(\pi^2).$
Let $\mathbf{V}$
be a subspace defined by the following
condition $v(u)>1/2.$
This space contains $\mathbf{Y}_{2,2}.$ 
Let $\mathcal{T}$
be the set of the singular residue classes in $\mathbf{Y}_{2,2}.$
For $T \in \mathcal{T},$ let $\mathbf{X}_T \subset T$
be the underlying affinoid.
We set
$${\mathbf{V}}_{1}:=\mathbf{V} \backslash \bigcup_{T \in \mathcal{T}}\mathbf{X}_T.$$

In this subsection, we write $\mathbf{V}'_{e_1}$
for
$\mathbf{W}_3$ in the previous subsection.
This space contains the space $\mathbf{Z}_{1,1,e_1}.$ 
Let $\mathcal{S}_{e_1}$
be the set of the singular residue classes in $\mathbf{Z}_{1,1,e_1}.$
For $S \in \mathcal{S}_{e_1},$ let $\mathbf{X}_S \subset S_{e_1}$
be the underlying affinoid.
We put
$$\mathbf{V}_{2,e_1}:=\mathbf{V}'_{e_1} \backslash \bigcup_{S \in \mathcal{S}_{e_1}} \mathbf{X}_S.$$

Let $(u,X_2,Y_2) \in \mathbf{V}_{3,e_1}$ be a subspace
defined by the following conditions;
$$0<v(u)<1/2,v(X_1)=\frac{1-v(u)}{q-1},v(Y_1)=\frac{v(u)}{q(q-1)}.$$
Then, the space $\mathbf{V}_{3,e_1}$ contains the space
$\mathbf{Y}_{3,1e_1}.$
Furthermore, for $i=2,3,$ let $\mathbf{V}_{i,e_1^{\vee}}$
be a subspace defined by the following condition;
$(u,X_2,Y_2) \in \mathbf{V}_{i,e_1^{\vee}}$
is equivalent to
$(u,Y_2,X_2) \in \mathbf{V}_{i,e_1}.$

Let $(u,X_2,Y_2) \in \mathbf{V}_{3,c}$\ (resp.\ $\mathbf{V}'_{2,c}$)
be a subspace defined by the following conditions;
$$v(X_1)=v(Y_1)=\frac{v(u)}{q(q-1)},v(X_2)=v(Y_2)=\frac{v(u)}{q^3(q-1)}.$$
 and $0<v(u)<1/2.$ (resp.\ $\frac{1}{q+1}<v(u)<\frac{q}{q+1}.$) 
Let $\mathcal{S}_c$
be a set of the singular residue classes of the space
$\mathbf{Z}_{1,1,c}.$
For $S \in \mathcal{S}_c,$
let $\mathbf{X}_S \subset S$ denote the underlying affinoid subdomain.
Then, we set
$$\mathbf{V}_{2,c}:=\mathbf{V}'_{2,c} \backslash \bigcup_{S \in \mathcal{S}_c}\mathbf{X}_S.$$

Let $\mathcal{C}(\pi^2)$
be a covering of $\mathcal{X}(\pi^2)$ consists of
$$\{\mathbf{V}_{1},\mathbf{V}_{i,e_1^{\vee}},\mathbf{V}_{i,e_1},\mathbf{V}_{i,c}\}_{i=2,3} \cup
\{\mathbf{X}_T,\mathbf{X}_{S_1},\mathbf{X}_{S_2}\}_{T \in \mathcal{T},S_1 \in \mathcal{S}_{e_1},
S_2 \in \mathcal{S}_c}.$$
\begin{proposition}
Let the notation be as above.
Then, the covering $\mathcal{C}(\pi^2)$
is a stable covering of $\mathcal{X}(\pi^2).$
\end{proposition}
\begin{proof}
Note that
the disjoint union
$(\mathbf{V}_1 \cap \mathbf{V}_{2,e_1^{(\vee)}}) \cup 
(\mathbf{V}_{2,e_1^{(\vee)}} \cap \mathbf{V}_{3,1,e_1^{(\vee)}})$ is 
equal to
$\mathbf{W}_{3}^{(\vee)} \backslash \mathbf{Z}_{1,1,e_1^{(\vee)}}.$
Hence,
the intersections
$\mathbf{V}_1 \cap \mathbf{V}_{2,e_1^{(\vee)}}$
and $\mathbf{V}_{2,e_1^{(\vee)}} \cap \mathbf{V}_{3,1,e_1^{(\vee)}}$
are disjoint unions of annuli by Proposition \ref{rop1}.3.
The intersection $\mathbf{V}_{2,c} \cap \mathbf{V}_1$
is equal to $\mathbf{U}_3$ and, hence
the intersection is a disjoint union of annuli by Proposition \ref{rop2}.2. 
The intersection $\mathbf{V}_{2,c} \cap \mathbf{V}_{3,c}$
is equal to $\mathbf{U}_2$ and, hence
is a disjoint union of annuli by Proposition \ref{rop2}.2.
The complements $\mathbf{V}_{3,e_1^{(\vee)}} 
\backslash \mathbf{Y}_{3,1,e_1^{(\vee)}}$
and $\mathbf{V}_{3,c} \backslash \mathbf{Y}_{3,1,c}$
are disjoint union of annuli by Propositons \ref{rop1}.1,2 and \ref{rop2}.1.
Hence, the required assertion follows.
\end{proof}

We explain a shape of the stable reduction of the Lubin-Tate space 
$\mathcal{X}(\pi^2)$ as already mentioned in Introduction.
Let $\overline{\mathbf{Y}}^c_{2,2}$ 
be the projective completion of the affine curve
 $\overline{\mathbf{Y}}^{\pi_2}_{2,2}.$
 Then, the complement 
 $\overline{\mathbf{Y}}^c_{2,2} \backslash 
 \overline{\mathbf{Y}}^{\pi_2}_{2,2}$
 consist of $(q+1)$ closed points.
The projective curve $\overline{\mathbf{Y}}^c_{2,2}$
meets the projective completion 
$\{\overline{\mathbf{Z}}^c_{1,1,\zeta}\}$
of $(q+1)$ 
affine curves $\{\overline{\mathbf{Z}}^{\pi_2}_{1,1,\zeta}\}$
at each infinity.
The complement $\overline{\mathbf{Z}}^c_{1,1,\zeta}
\backslash \overline{\mathbf{Z}}^{\pi_2}_{1,1,\zeta}$
consists of two closed points.
The projective curve
$\overline{\mathbf{Z}}^c_{1,1,\zeta}$
meets the projective completion 
$\overline{\mathbf{Y}}^c_{3,1,\zeta}$
of $(q+1)$ affine curves $\overline{\mathbf{Y}}_{3,1,\zeta}$
at each infinity.
The curve $\overline{\mathbf{Y}}^c_{3,1,\zeta}$
meets the Igusa curve ${\rm Ig}(p^2)$ at each infinity.
Since the affine curve $\overline{\mathbf{Y}}^{\pi_2}_{3,1,\zeta}$
has $(q+1)$ infinity points,
there exist $q(q+1)$ Igusa curves ${\rm Ig}(p^2)$
in the stable reduction of the Lubin-Tate space $\mathcal{X}^{\pi_2}(\pi^2).$

\subsection{Intersection Data}
We include the intersection multiplicities 
in $\mathcal{X}(\pi^2)$
below in Table 1.
These numbers have been obtained via a rigid analytic reformulation.
Let $C$ be a projective smooth curve over a non-archimedean local field
$F.$ We assume that $C$ admits a semi-stable model $\mathcal{C}$ 
over some extension $E/F.$ 
Suppose that $X$ and $Y$ are irreducible components
of $\mathcal{C},$
and that they intersect in an ordinary double point $P.$
Then ${\rm red}^{-1}(P)$
is an annulus, say with width $w(P).$
Let $e_p(E)$ denote the ramification index of $E/F.$
In this case, the intersection multiplicity
of $X$ and $Y$ at $P$ can be found
by
$$M_E(P):=e_p(E) \cdot w(P).$$
Note that while intersection multiplicity depends on $E,$
the width makes sense even over $\mathbf{C},$
which in some sense makes width a more natural invariant from
the purely geometric perspective as mentioned in
\cite[Section 9.1]{CM}.
Now, for our calculation of $M_E(P)$ on $\mathcal{X}(\pi),$
we take $e_p(E)=(q^2-1).$
Then, we have $({\rm Ig}(p),\overline{\mathbf{Y}}^c_{1,1})=1.$
Now, for our calculation of $M_E(P)$ on $\mathcal{X}(\pi^2),$
we take $e_p(E)=q^4(q^2-1).$

\begin{minipage}{.4\textwidth}
\begin{tabular}{|c|c|c|c|c|c|}
\hline  $P$ & 
\scriptsize{$({\rm Ig}(p^2),\overline{\mathbf{Y}}^c_{3,1,e_1})$} & 
\scriptsize{$(\overline{\mathbf{Y}}^c_{3,1,e_1},\overline{\mathbf{Z}}^c_{1,1,e_1})$} & 
\scriptsize{$(\overline{\mathbf{Z}}^c_{1,1,e_1},\overline{\mathbf{Y}}^c_{2,2})$}  &
\scriptsize{$(\overline{\mathbf{Z}}^c_{1,1,e_1},\mathbf{X}_S)$}&
\scriptsize{$(\overline{\mathbf{Y}}^c_{2,2},\mathbf{X}_{T})$}
\\
\hline 
\rule[-13pt]{0pt}{31pt}
$w(P)$ & $\frac{1}{q^3(q^2-1)}$ & $\frac{1}{2q^4(q+1)}$ & 
$\frac{1}{2q^4(q+1)}$ & $\frac{1}{4q^4}$ & $\frac{1}{q^3(p+1)}$
\\
\hline \rule[-13pt]{0pt}{31pt} $M_E(P)$ & 
$q$ & 
$\frac{q-1}{2}$ &  
$\frac{q-1}{2}$ &
$\frac{q^2-1}{4}$ &  $q(q-1)$ \\ 
\hline 
\end{tabular}
\end{minipage}

\begin{center}
Table 1: Intersection Multiplicity Data for $\mathcal{X}(\pi^2)$ 
\end{center}

\end{document}